\documentclass{amsart}
\usepackage{hyperref}
\usepackage{amsmath,amsthm, amsfonts,amssymb}
\usepackage[english]{babel}
\usepackage{appendix} 
\usepackage{cancel}
\usepackage[all]{xy}
\usepackage[colorinlistoftodos,shadow]{todonotes}
\usepackage{setspace}
\usepackage[utf8]{inputenc}

\usepackage{tikz}
\usetikzlibrary{mindmap,trees}
\usepackage{verbatim}
\usepackage{nag}

\usepackage{fixltx2e}
\let\cite\cite
   \usepackage[numbers,square]{natbib}

\newcommand{\bb}[1]{\mathbb{#1}}
\newcommand{\fk}[1]{\mathfrak{#1}}
\newcommand{\mcal}[1]{\mathcal{#1}}

\newcommand{\C}{\bb C}
\newcommand{\Z}{\bb Z}
\newcommand{\Q}{\bb Q}
\newcommand{\N}{\bb N}
\newcommand{\F}{\bb F}

\newcommand{\OO}{\mcal O}

\newcommand{\Cp}{{\bb{C}}_{p}}
\newcommand{\Zp}{{\bb{Z}}_{p}}
\newcommand{\Qp}{{\bb{Q}}_{p}}

\newcommand{\set}[1]{\left\{{#1}\right\}}
\newcommand{\abs}[1]{\left|{#1}\right|}
\newcommand{\ot}{\otimes} 
\newcommand{\xto}{\xrightarrow}  
\newcommand{\To}{\longrightarrow}
 
\newcommand{\oln}[1]{\overline{#1}}
\newcommand{\wh}[1]{\widehat{#1}}

 \newcommand{\etl}[3]{H^{{#1}}_{\et}({#2},{#3})}

\def\Frac{\mathop{\mbox{\normalfont Frac}}\nolimits}

\def\spec{\mathop{\mbox{\normalfont Spec}}\nolimits}

\def\sing{\mathop{\mbox{\normalfont sing}}\nolimits}
\def\Ext{\mathop{\mbox{\normalfont Ext}}\nolimits}

\def\Rep{\mathop{\mbox{\normalfont\textbf{Rep}}}\nolimits}
\def\Gal{\mathop{\mbox{\normalfont Gal}}\nolimits}

\def\Fil{\mathop{\mbox{\normalfont Fil}}\nolimits} 

\def\Res{\mathop{\mbox{\normalfont Res}}\nolimits} 
\def\et{{\text{\'et}}}
\def\dR{{\text{dR}}}
\def\cris{{\text{cris}}}
\def\st{{\text{st}}}

\newtheorem{theorem}{Theorem}[section]
\newtheorem{proposition}[theorem]{Proposition}

\newtheorem{corollary}[theorem]{Corollary}
\theoremstyle{definition}
    \newtheorem{definition}[theorem]{Definition}
    \newtheorem{example}[theorem]{Example}

    \newtheorem{strategy}[theorem]{Strategy}
    
\theoremstyle{remark}
    \newtheorem{remark}[theorem]{Remark}

\begin{document}

\title{A Kulikov-type classification theorem for a one parameter family of K3-surfaces over a $p$-adic field and a good reduction criterion}


\author{Jesús Rogelio Pérez Buendía}
\address{CONACYT-CIMAT Mérida}
\email{rogelio.perez@cimat.mx}



\keywords{}

\date{July 2018}


\maketitle

\begin{abstract}
\label{sec:abstract}
In this paper, we prove a $p$-adic analogous of the Kulikov classification theorem~~\cite{Kulikov:1977aa,Persson:1981wp,MR0466149} for the central fiber of a degeneration of $K3$-surfaces in terms of the nilpotency degree of  the monodromy of the family.

Namely, let $X_K$ be a be a smooth, projective $K3$-surface over the $p$-adic field $K$, which has either, a minimal semistable model $X$ over $\mathcal O_K$, or combinatorial reduction. 
If we let $N_{st}$ be the Fontaine's monodromy operator on $D_{st}(H^2_{\et}(X_{\oln K},\Qp))$, then we prove that the degree of nilpotency of $N_{st}$ determines the type of the special fiber of $X$. As a consequence we give a criterion for the good reduction of the semistable $K3$-surface $X_K$  in terms of its $p$-adic representation $H^2_{\et}(X_{\oln K},\Qp)$, which is similar to the criterion of good reduction for $p$-adic abelian varieties and curves given by~\cite{Coleman:1999aa} and~\cite{Andreatta:2015aa}.

\end{abstract}

\section{Introduction} 
\label{sec:introduction}

We fix once and for all a prime number $p>3$. If $K$ is a field, we denote by $G_K=\Gal(\oln K/K)$ its absolute Galois group. 

The main object of study of this paper is the interplay between the geometry of algebraic varieties and their cohomology. In general it is known that the geometry of an algebraic variety over a field determines the various cohomology groups with their extra structure. For example if $X$ is a smooth, proper algebraic variety over the complex numbers $\C$, then the Hodge structure on its Betti cohomology is pure with determined weights. 
Similarly, if $X$ is a smooth, proper algebraic variety over a $p$-adic  field $K$\index{$p$-adic  field $K$}, then its $p$-adic   \'etale cohomology groups $V_i:= H^i_{\et}(X_{\oln K},\Qp)$  are $p$-adic  $G_K$-representations whose type is determined by the geometry of various integral models of $X$. For instance if $X$ has good reduction then the $V_i$'s are crystalline $G_K$ representations~\cite{MR1463696}, if $X$  has semistable reduction, then the $V_i$'s are semistable representations~\cite{MR2372150}, etc.

In general it is not true that the cohomology groups of an algebraic variety determine their geometric properties, however, for certain very special classes of varieties it has been known for some time that this might happen.

Here are some examples:

For Abelian varieties over $\C$,  we have the Torelli type theorem (see~\cite[ch 2]{Rohde2009}):

\begin{theorem}[Riemann]\label{thm:ianTorelli}
An abelian variety over $\C$ is determined by its periods. More precisely, if $A, A'$ are complex polarized abelian varieties, and we have an isomorphism of Hodge structures: $$\phi:H^1(A,\Z)\to H^1(A',\Z),$$ then the abelian varieties $A$ and $A'$ are isomorphic.
\end{theorem}

Moreover, if $A$ is an Abelian variety over a $p$-adic field $K$ we have the following theorem~\cite{MR1463696}\cite[theorem II.4.7]{Coleman:1999aa}\cite[corollary 5.3.4]{MR1804530}:
\begin{theorem}[Faltings, Coleman-Iovita, Breuil]\label{thm:GoodReductionIovitaColemanB}
	Let $H=H^1_{\et}(A_{\oln K},\Qp)$.
    \begin{itemize}
        \item $A$ has good reduction if and only if $H$ is a crystalline $G_K$-representation.
        \item $A$ has semistable reduction if and only if  $H$ is a semistable $G_K$-representation.
    \end{itemize}
\end{theorem}

The class of $K3$-surfaces is another very interesting class of algebraic varieties which resembles the class of Abelian varieties. More precisely, they satisfy a Torelli theorem~\cite{Looijenga:1980vy}:

    \begin{theorem}[Weak Torelli Theorem]\label{thm:TorelliK3Weak}
    Two complex $K3$-surfaces $X$, $X'$ are isomorphic if and only if there exists an isometry from $H^2(X,\Z)$ to $H^2(X',\Z)$ which sends $H^{2,0}(X,\C)$ to $H^{2,0}(X',\C)$. 
    \end{theorem}

Also, if $\mathcal X\to \Delta$ is a degeneration of $K3$-surfaces over the open unit complex disk $\Delta$, we have the important theorem of~\cite[theorem II]{Kulikov:1977aa} and~\cite{Persson:1981wp,MR0466149} (see also theorem \eqref{thm:KulikovClassification}).

    \begin{theorem}[Kulikov-Persson-Pinkham]
    Let $\pi:\mathcal X \to \Delta$ be a semistable degeneration of $K3$-surfaces with all components of the central fiber $\mathcal X_0=\pi^{-1}(0) = \bigcup V_i$ algebraic. Let $N=\log T:H^2(\mathcal X_t,\Z)\to H^2(\mathcal X_t,\Z)$ be the monodromy operator. After birational modification we may assume that $\pi:\mathcal X\to \Delta$  is a Kulikov model (definition~\eqref{def:kulikovmodel}). Then the central fiber $\mathcal X_0$ is one of the following:
    \begin{enumerate}
        \item[I.] $\mathcal X_0$ is a $K3$-surface and $N=0$.
        \item[II.] $\mathcal X_0=V_0\cup V_1\cdots V_r$, where $V_0,V_r$ are smooth rational, and $V_1, \ldots, V_{r-1}$ are smooth elliptic ruled and $V_i\cap V_j\ne \emptyset$ if and only if $j=i\pm 1$. $V_i\cap V_j$ is then a smooth elliptic curve and a section of the ruling on $V_i$, if $V_i$ is elliptic ruled. $N\ne0$ but $N^2=0$.
        \item[III.] $\mathcal X_0=\cup V_i$, with each $\oln V_i$ smooth rational and all double curves are cycles of rational curves. The dual graph $\Gamma$ is a triangulation of $S^2$. In this case $N^2\ne 0$, but $N^3=0$.
    \end{enumerate}
    \end{theorem}

\begin{remark}\label{rem:ZeroMonodromy}
Note that in particular $\mathcal X_0$ is smooth if and only if $N=0$.
\end{remark}
   
 \begin{remark}\label{p-adicfield}
 	Fix a prime number $p>3$ and let now $K$ be a $p$-adic field, that is, a complete discrete valuation field of characteristic 0, with valuation ring $\mathcal O_K$ and perfect residue field $k$ of characteristic $p$. Typically $K$ will be a finite extension of $\Qp$.  
 \end{remark}

In this paper we prove the following theorem (section~\ref{sec:the_main_theorem_v2} theorem~\eqref{thm:MainTheorem}):

\begin{theorem}
Let $X_K$ be a smooth projective $K3$-surface over $K$, and let $X\to\spec{\mathcal O_K}$ be a semistable minimal model of $X_K$ (or we can just assume that $X_K$ has combinatorial reduction as in definition~\eqref{def:sncl-reduction}). Let ${\oln X}$ be the special fiber of $X$. We denote $D_{st}=D_{st}( H^2_{\et}(X_{\oln K},\Qp))$ and let $N_{st}:D_{st}\to D_{st}$
be the monodromy operator on $D_{st}$. Then we have $3$ possibilities for the special fiber ${\oln X}$, distinguished in terms of the nilpotency degree of the monodromy operator $N_{st}$, as follows:
\begin{itemize}
    \item[I.]  $N_{st}=0$ if and only if
     ${\oln X}$ is a nonsingular $K3$ surface.
    \item[II.]  $N_{st}\ne 0$ but $N^2_{st}=0$ if and only if
     ${\oln X}= \cup_{i=1}^n V_i$ where the $V_i$ are rational surfaces and $V_2, \ldots, V_{n-1}$ are elliptic ruled surfaces.
    \item[III.] $N_{st}^2\ne 0$ but $N_{st}^3=0$ if and only if
     ${\oln X}= \cup_{i=1}^n V_i$ where all the $V_i$ are rational surfaces.
\end{itemize}
That is, the type of the combinatorial special fiber $\oln X$ can be distinguish in terms of the monodromy operator $N_{st}$ in exactly the same way as in \cite{MR0466149} and \cite{Kulikov:1977aa} for the complex case.
\end{theorem}

We prove it using a novel method that can be applied to a more general class of varieties as follows:

\begin{strategy}
Let us suppose that $\mathcal A$ is a class of varieties over various fields satisfying the following two conditions:
	\begin{enumerate}
	    \item[C.I] If $X$ is a scheme over $\mathcal O_K$ such that its generic fiber $X_K$ is a smooth, proper variety in $\mathcal A$ and its special fiber $\oln X$ is a semistable variety over $k:=\mathcal O_K/\fk m_K$, then $\oln X$, a logarithmic scheme (with the natural logarithmic structure), has global deformations over $W(k)[[t]]$ of the type described in theorem~\eqref{thm:liting of semistable K3}.
	    \item[C.II] If $Y$ is a family of varieties in $\mathcal A$ over the complex open unit disk $\Delta $, degenerating exactly at $0$, then there is a Kulikov-type theorem saying that: the family is smooth if and only if the monodromy operator of the logarithmic cohomology of its special fiber vanishes, or more generally, if the monodromy operator classify the type of the generic fiber.

	    Then, following the same steps as in section~\eqref{sec:the_main_theorem_v2} below, one would be able to prove a similar classification theorem of type (C.II), for a variety $X_K$ in $\mathcal A$ over a $p$-adic field $K$ as in (C.I).
	\end{enumerate}
\end{strategy}

In this paper we prove that indeed the class of $K3$-surfaces satisfies these properties, and then we conclude the main theorem~\eqref{thm:MainTheorem} in page~\pageref{thm:MainTheorem}.  

As a corollary (theorem~\eqref{thm:MainTheoremCrystalline}) we give a good reduction criterion for $K3$-surfaces:

\begin{theorem}\label{thm:MainIntro}
The $K3$- surface $X_K$ has good reduction if and only if $H^2_{\et}(X_{\oln K},\Qp)$ is a crystalline $G_K$-representation.
\end{theorem}

This theorem~\eqref{thm:MainTheoremCrystalline} generalizes the analogous result of \cite[theorem II.4.7]{Coleman:1999aa} for abelian varieties. Which is also analogous to the result of \cite[theorem 1.6]{Andreatta:2015aa} for curves below. 

As we have mentioned before, the cohomology does not always determine the geometry of the algebraic varieties. For example,
it is known that the geometry of curves is not determined by the structure of their cohomology groups. Nevertheless, their geometry is determined by the quotients of their unipotent fundamental groups as follows~\cite{Andreatta:2015aa}:

Let $K$ be a finite extension of $\Qp$ and suppose that $X_K$ is a curve with semistable reduction. Assume also that the genus of $X_K$ is larger or equal to $2$. For a fix geometric point $b$ of $X_K$ let, for every prime $l$,  $\pi^{(l)}$ be the maximal pro-$l$ quotient of the geometric  \'etale fundamental group $\pi_1(X_{\oln K},b)$ of $X_K$ and let $\set{\pi_1^{(l)}[n]}_{n\geq 1}$ be  the lower central series of $\pi_1^{(l)}$.  

\begin{theorem}[Oda]\label{thm:Oda'stheorem}\cite[theorem 3.2]{MR1348768}
$X_K$ has good reduction if and only if for some prime integer $l\ne p$ the outer representations $\pi_1^{(l)}/\pi_1^{(l)}[n]$ are unramified for all $n>1$.
\end{theorem}

The following theorem is a $p$-adic  analogue of theorem of Oda above.

\begin{theorem}[Andreatta-Iovita-Kim]
	\label{thm:Iovita-Adrquotienteatta-Kimm}
	\cite[theorem 1.6]{Andreatta:2015aa}
If $G^{\et}$ denotes the unipotent $p$-adic   \'etale fundamental group of $X_{\oln K}$ for the base point $b$, then $X_K$ has good reduction if and only if $G^{\et}$ is a crystalline $G_K$-representation.
\end{theorem}

 This raises the very interesting question: given a class of algebraic varieties, are there combinatorial (linear algebra type) objects attached to them which determine their geometry? If yes, what are they? In this paper we give an answer to this question for a family of $K3$-surfaces. 
 
 \begin{remark}
 	\label{rem:otherprofs}
 	 We proved the classification theorem~\eqref{thm:MainTheorem} and its corollary (the good reduction criterion)~\eqref{thm:MainTheoremCrystalline} in my PhD thesis (A Crystalline Criterion for Good Reduction on Semi-stable $K3$-Surfaces over a $p$-Adic Field,  Concordia University, Montreal, Canada. 2014). The good reduction criterion corollary~\eqref{thm:MainTheoremCrystalline} was also studied later using different methods and with some variants in the hypothesis by~\cite{MR3699071,MR3299851} and by Hernández-Mada in his Ph.D thesis (2015); and recently generalized by \cite{Chiarellotto:2016ca}.
 	The novelty of our approach to prove the $p$-adic good reduction criterion, is the fact that we  construct a one parameter logarithmic degeneration of $K3$-surfaces on p-adic settings that mimics the degeneration that occurs over the complex numbers (the Kulikov degeneration of K3 surfaces over the complex disc), and that gives a characterization; using logarithmic geometry, $p$-adic Hodge theory and  complex methods; of the type of the special fiber in terms of the Fontaine's monodromy operator $N_{st}$ on $D_{st}(H_{\text{\'et}}^{2}(X_{\overline{K}}, \mathbb{Q}_p))$. 
 \end{remark}
 
\section{Kulikov Degeneration's Theorem} 
\label{sec:kulikov_degeneration_s_theorem}

We will briefly describe  the Kulikov-Persson-Pinkham's classification theorem of the central fiber of a semistable family of complex $K3$-surfaces over the complex open disk. The main references are~\cite{Persson:1981wp,MR0466149,Nishiguchi:1983wa,Morrison:1984uh,Friedman:1986ve,Kulikov:1977aa}.

Denote by $\Delta:=\set{z\in \C : \abs z < \varepsilon}$ the open small disk and by $\Delta^*$ the punctured disk, that is $\Delta^*=\Delta\setminus\set 0$.

\begin{definition}\label{def:degeneration}
A degeneration of $K3$-surfaces is a flat proper holomorphic  map $\pi: \mathcal Y \to \Delta$ of relative dimension $2$ such that $\mathcal Y_t:=\pi^{-1}(t)$ is a smooth $K3$-surface for $t\ne0$. We call $\mathcal \mathcal Y_0:=\pi^{-1}(0)$ the degenerated fiber or central fiber. We assume that $\mathcal Y$ is K\"ahler.

If we have a fixed $K3$-surface $Y$, then a degeneration of $Y$ is a degeneration of $K3$-surfaces such that for some $t\ne 0$, $\mathcal Y_t=Y$.
\end{definition}

\begin{definition}\label{def:semistabledegeneration}
A degeneration $\pi:\mathcal Y \to \Delta$ is semistable if the central fiber $\mathcal \mathcal Y_0$ is a reduced divisor with normal crossings, that is the union $\mathcal \mathcal Y_0=\cup Y_i$ of irreducible components with each $Y_i$ smooth and the $Y_i$'s meeting transversally so that locally $\pi$ is defined by an equation of the form $0=x_0x_2\ldots x_k$ for some $k$.
\end{definition}

\begin{definition}\label{def:modification}
A degeneration $\pi':\mathcal Y'\to \Delta$ is called a {modification} of a degeneration $\pi:\mathcal Y\to \Delta$; if there exists a birational map $\psi:\mathcal Y\to \mathcal Y'$ such that the diagram commutes:
$$
\xymatrix{
\mathcal Y\ar@{.>}[rr]^{\psi}\ar[dr]_{\pi} & &  \mathcal Y' \ar[dl]^{\pi'} \\
& \Delta &
}
$$

and the restriction of $\phi$  to  $\pi^{-1}(\Delta^*)$ gives an isomorphism $\pi^{-1}(\Delta^*)\xto \phi \pi'^{-1}(\Delta^*)$ over $\Delta^*$.
\end{definition}


\begin{definition}\label{def:kulikovmodel}
A semistable degeneration of $K3$-surfaces, $\pi:\mathcal Y \to \Delta$,  with trivial canonical bundle, $K_Y\equiv 0$, is called a \emph{Kulikov model} or a \emph{good model}.
\end{definition}

We have the following theorem of \cite{Kulikov:1977aa} and~\cite{Persson:1981wp,MR0466149} (see also~\cite[theorem 0.1]{Friedman:1986ve}):

\begin{theorem}[Kulikov-Persson-Pinkham]\label{thm:Persson-Pinkham-Kulikov} \cite[theorem II]{Kulikov:1977aa} 
	
Let $\pi: \mathcal Y \to \Delta$ be a semistable degeneration of $K3$-surfaces such that all components of the special fiber are algebraic. Then there exists a modification $\pi':\mathcal Y\to \Delta$ of $ \pi:\mathcal Y \to \Delta$ which is a Kulikov model.
\end{theorem}

Given a Kulikov model, Kulikov-Persson-Pinkham~~\cite{Kulikov:1977aa,MR0466149,Persson:1981wp} give a description of the cohomology of its special fiber in terms of the monodromy operator acting on cohomology.

Let $\pi:\mathcal Y\to \Delta $  be a degeneration of $K3$-surfaces, and let $\pi^*:\mathcal Y^*\to \Delta^*$ be the restriction to the punctured disk. Fix a smooth fiber  $Y:=\mathcal Y_t$, which is a $K3$-surface.  Since $\pi^*$ is a fibration, the fundamental group of $ \Delta^*$ acts on the cohomology $H^2(Y,\Z)$.

\begin{definition}\label{def:Picard-LefschetzTransformation}
The map \(
T:H^2(Y,\Z)\To H^2(Y,\Z),
\) induced by the action of $\pi_1( \Delta^*)$ is called the \emph{Picard-Lefschetz transformation}.
\end{definition}

We have the following theorem of Landman~\cite{Landman:2010wp}:
\begin{theorem}\label{thm:MonodromyTheorem}\phantom{.}\hspace{0.1cm}
	
\begin{itemize}
    \item $T$ is quasi-unipotent, with index of unipotency at most $2$. In other words, there is some $k$ such that $$ (T^k-I)^{3}=0.$$
    \item If $\pi:\mathcal Y\to \Delta$ is semistable, then $T$ is unipotent, that is $k=1$.
\end{itemize}
\end{theorem}

Therefore, for a Kulikov model of a $K3$-surface, we have that the Picard-Lefschetz transformation is unipotent. Moreover, we can define the logarithm of $T$ (in the semistable case) by:

\begin{definition}\label{def:MonodromyOperator}
The \emph{Monodromy operator} $N$ on $H^2(Y,\Z)$ is defined as:
\[
N:=\log T = (T-I)-\frac 12 (T-I)^2.
\]
$N$ is nilpotent, and the index of unipotency of $T$ coincides with the index of nilpotency of $N$; in particular, $T=I$ if and only if $N=0$.
\end{definition}

We finally present the  very important classification theorem of \cite{Kulikov:1977aa}, \cite{Persson:1981wp,MR0466149} of the central fiber of a good model:

\begin{theorem}[Kulikov-Persson-Pinkham]\label{thm:KulikovClassification}\cite[theorem II]{Kulikov:1977aa} \cite[sec 3.3]{MR0466149}
Let $\pi:\mathcal Y \to \Delta$ be a semistable degeneration of $K3$-surfaces with all components of the central fiber $\mathcal Y_0=\pi^{-1}(0) = \bigcup V_i$ algebraic.

 Let $N=\log T:H^2(\mathcal Y_t,\Z)\to H^2(\mathcal Y_t,\Z)$ be the monodromy operator. After birational modifications we may assume that $\pi:\mathcal Y\to \Delta$  is a Kulikov model. Then the central fiber $\mathcal Y_0$ is one of the following:
\begin{enumerate}
    \item[I.]  $\mathcal Y_0$ is a $K3$-surface and $N=0$.
    \item[II.] $\mathcal Y_0=V_0\cup V_1\cdots V_r$, where $V_0,V_r$ are smooth rational, and $V_1, \ldots, V_{r-1}$ are smooth elliptic ruled and $V_i\cap V_j\ne \emptyset$ if and only if $j=i\pm 1$. $V_i\cap V_j$ is then a smooth elliptic curve and a section of the ruling on $V_i$, if $V_i$ is elliptic ruled. $N\ne0$ but $N^2=0$.
    \item[III.] $\mathcal Y_0=\cup V_i$, with each $V_i$ smooth rational and all double curves are cycles of rational curves. The dual graph $\Gamma$ is a triangulation of $S^2$. In this case $N^2\ne 0$, but $N^3=0$.
\end{enumerate}
\end{theorem}

The proof of these results uses the Clemens-Schmid exact sequence. An account of this sequence is the paper~\cite{Morrison:1984uh} (see also~\cite[theorem 0.2]{Friedman:1986ve}) in which as application we have the proof of the monodromy criterion in previous theorem~\eqref{thm:KulikovClassification}.


\section{Logarithmic Structures} \label{sec:logarithmic_structures}
In this section we give a fast review of the theory of logarithmic structures of Fontaine-Illusie-Kato~\cite{Kato:1989vk}. A complete treatment of the subject can be found at~\cite{Ogus:2006tu}.
\begin{definition}
\label{def:monoid}
A \emph{monoid} is a commutative semi-group with a unit. A morphism of monoids is required to preserve the unit element.  We denote by \textbf{Mon} the category of monoids.
\end{definition}
\begin{definition}
Let $X$ be a scheme. A \emph{pre-logarithmic structure} on $X$ is a sheaf of monoids $M_X$ (on the étale or Zariski site of $X$) together with a morphism of sheaves of monoids: $\alpha:  M_X\To \OO_X$, called the \emph{structure morphism}, where we consider $\OO_X$ a monoid with respect to the multiplication.

A pre-logarithmic structure is called a \emph{logarithmic structure} if $\alpha^{-1}(\OO_X^*) \simeq \OO_X^*$ via $\alpha$.

The pair $(X,M_X)$ is called a \emph{logarithmic scheme} or a \emph{logarithmic scheme}, and it will be denoted by $X^{\log}$ if the logarithmic structure is clear.
\end{definition}

We have a functor $i$ from the category of logarithmic structure of $X$ to the category of pre-logarithmic structure of $X$, by sending a logarithmic structure $M$ in $X$ to itself; considered as a pre-logarithmic structure $i(M)$. Vice-versa, given a pre-logarithmic structure, we can construct a logarithmic structure  $M^{a}$ out of it, in such a way that $(\ )^{a}$ is left adjoint of $i$, so $M^a$ is \emph{universal}~\cite[1.3]{Kato:1989vk}.
\begin{remark}
The category of schemes is a full subcategory of the category of logarithmic schemes. Indeed, given a scheme $X$ the trivial inclusion $\OO_X^*\To \OO_X$ gives the trivial logarithmic structure on $X$, which is, in fact, an initial object in he category of logarithmic structure over $X$.  Also we have the identity map $\OO_X\To \OO_X$ which gives a different logarithmic structure on $X$, which is in fact a final object.\end{remark}
To clarify  the action of this inclusion on morphisms we need the following definitions.
\begin{definition}
Let $f:X\to Y$ be a morphism of schemes. Given a logarithmic structure $M_Y$ on $Y$ we can define a logarithmic structure on $X$, called \emph{the inverse image} of $M_Y$, to be the logarithmic structure associated to the pre-logarithmic structure
$$f^{-1}(M_Y)\to f^{-1}(\OO_Y)\to \OO_X,$$
it is denoted by $f^*(M_Y)$.
\end{definition}

\begin{definition}
A morphism of logarithmic schemes $(X, M_X)\To (Y,M_Y)$  consists of a morphism of underlying schemes $f:X\to Y$ and a morphism $f^b: f^*M_Y\to M_X$ of logarithmic structure on $X$.
\end{definition}

One of the main examples of interest for us is the following:
\begin{example}
 Let $X$ be a regular scheme. Let $D$ be a divisor of $X$. We can define a logarithmic structure $M$ on $X$ associated to the divisor $D$ as
\[M(U):=\set{g\in \OO_X(U):g|_{U\setminus D}\in \OO_X^*(U\setminus D)}.
\]
\end{example}

Let $P$ be a monoid and $R$ a ring and let $R[P]$ be its monoid algebra. 
The natural inclusion $P\To R[P]$, induces a pre-logarithmic structure on $P\to \spec{R[P]}$.
The associated logarithmic structure is called the \emph{canonical logarithmic structure} on $\spec{R[P]}$. 
The logarithmic structure on $\spec{R[P]}$ is the inverse image of the logarithmic structure on $\spec{\Z[P]}$ under the natural map $\spec{R[P]}\To \spec{\Z[P]}$.
\begin{definition}\label{def:logChart}
Let $(X,M_X)$ be a logarithmic scheme  and $P$ be a monoid. Denote by $P_X$ the constant sheaf associated to $P$. A \emph{chart} for $M_X$ is a morphism $P_X\to M_X$ such that we have an isomorphism between the logarithmic structures \( P^a\to M_X \) where $P^a$ is the logarithmic structure associated to the pre-logarithmic structure given by the map $P_X\to M_X\to\mathcal O_X$. Equivalently a chart  of $M_X$ is a morphism $X\to \spec{\Z[P]}$ of logarithmic structures, such that its morphism of logarithmic structures, $P_X\to M_X$, is an isomorphism.
\end{definition}

\begin{definition}\label{def:chartmorphisms}
Let $f:X\To Y$ be a morphism of logarithmic schemes. Consider the constant sheaves $P_X$ and $Q_Y$ on $X$ and $Y$ associated to the monoids $P$ and $Q$ respectively.  A \emph{chart} for the morphism $f$ is the data $(P_X\to M_X, Q_Y\to M_Y, Q\to P)$ such that:
\begin{itemize}
    \item The maps $P_X\to M_X$ and $Q_Y\To M_Y$ are charts of $M_X$ and $M_Y$.
    \item We have a commutative diagram:
    $$\xymatrix{ Q_X \ar[r]\ar[d] & P_X \ar[d] \\
                f^*M_Y \ar[r] & M_X }$$
\end{itemize}
where the top arrow is induced by the map $Q\to P$.
\end{definition}

Remember that given a monoid $P$, we can associate to it an abelian group denoted by $P^{gp}$ (its Grothendieck group). 

We have a canonical map $P\to P^{pg}$. This group satisfies the universal property that any morphism of monoids from $P$ to an abelian group $G$, factors trough  $P^{gp}$ in a unique way.
\begin{definition}\label{def:integral-saturated}
A monoid $P$ is called \emph{integral} if the canonical map $P\to P^{gp}$ is injective. It is called \emph{saturated} if it is integral and for any $p\in P^{gp}$, if $np\in P$ for some positive integer $n$, then $p\in P$.
\end{definition}

\begin{definition}\label{def:finescheme}
A logarithmic scheme $(X, M_X)$ is said to be \emph{fine}, if (\'etale) locally there is a chart $P\to M_X$, with $P$ a finitely generated integral monoid.

The scheme $(X,M_X)$ is \emph{fine and saturated} (fs) if $P$ is also saturated.  Equivalently a logarithmic scheme is fs if for any geometric point $\tilde x\to X$ the monoid $M_{\tilde x, X}$ is finitely generated and saturated.

If moreover $P\simeq N^r$ for some $r$, then we say that the logarithmic structure is \emph{locally free}.
\end{definition}

\begin{definition}\label{def:strictMorphism}
A morphism of logarithmic schemes $f:(X,M_X)\to (Y,M_Y)$ is called \emph{strict}, if  the morphism on logarithmic structures $f^*M_Y\to M_X$ is an isomorphism.
\end{definition}

\begin{definition}\label{def:closedimmersion}
A morphism of logarithmic schemes $\imath:(X,M_X)\to (Y, M_Y)$ is called a \emph{closed immersion} (resp. an exact closed immersion) if the underlying morphism of schemes $X\to Y$ is a closed immersion and $i^* M_Y\to M_X$ is surjective (resp. an isomorphism).
\end{definition}
\begin{definition}\label{def:logsmooth}
A morphism $f:X\to Y$ of fine logarithmic schemes is \emph{logarithmic smooth} (respectively logarithmic \'etale) if  \'etale locally  (on $X$ and $Y$) $f$ admits a chart $$(P_X\to M_X, Q_Y\to M_Y, Q\to P),$$ such that:
\begin{itemize}
    \item The kernel and the torsion part of the cokernel (resp. the kernel and the cokernel) of $Q^{gp}\to P^{gp}$ are finite groups of order invertible on $X$.
    \item The induced morphism of logarithmic schemes $$(X, M_X)\To (Y, M_Y) \times_{\spec{\Z[Q]}} \spec{\Z[P]}$$ is  \'etale in the classical sense.
\end{itemize}
\end{definition}

In logarithmic geometry~\cite{Ogus:2006tu,Olsson:2003uq,Kato:1989vk} we also have the notion of sheaves of logarithmic differentials that, as in the classical case, give us information about the smoothness (logarithmic smoothness) as we can see in the following proposition: 

\begin{proposition}\label{prop:log-smooth differentials}
Let $X\xto f Y\xto g Z$ be a morphism of logarithmic schemes. Consider the sheaves of logarithmic differentials $\omega_{Y/Z}^1$, $\omega_{X/Y}^1$ and $\omega_{X/Z}^1$.
Then we have the following:
\begin{enumerate}
    \item \label{part1} The sequence $f^* \omega_{Y/Z}^1\To \omega_{X/Z}^1 \To \omega_{X/Y}^1 \To 0$ is exact.
    \item \label{part2} If $f$ is logarithmic smooth, then $\omega_{X/Y}^1$ is a locally free $\mathcal O_X$-module. Moreover the sequence 
    \[
        0 \To f^*\omega_{Y/Z}^1 \To \omega_{X/Y}^1\To \omega_{X/Z}^1 \To 0
    \] 
    is exact.
    \item If $g\circ f$ is logarithmic smooth and the sequence in part~\eqref{part2} is exact and splits locally, then $f$ is logarithmic smooth.
\end{enumerate}
\end{proposition}

\begin{proof}
\cite[IV sec  3.2]{Ogus:2006tu} or \cite[sec 3]{Kato:1989vk}.
\end{proof}

\section{Simple Normal Crossing Log $K3$-Surfaces} 
In this section we are mainly following the work of \cite{Nakkajima:2000tl}.
We are assuming  that all schemes are noetherian and that all morphisms are of finite type.
\begin{definition}
\label{Definition: NCL}
Let $k$ be a field. A \emph{normal crossing variety} $Y/k$ over $k$ is a geometrically connected scheme  $Y$ over $k$, whose irreducible components are geometrically irreducible  and of the same dimension $d$, and such that $Y$ is isomorphic to $\spec{k[x_0, \ldots , x_d ] / (x_0\cdots x_r)}$ étale locally over $Y$, where $0\leq r\leq d$ is a natural number that depends on étale neighborhoods.
\end{definition}
We denote by $Y_{\sing}$ the \emph{singular locus}  of $Y$.
So $Y_{\sing} := D_1\cup D_2\cup\cdots \cup D_m$
is a disjoint union of the $m$ connected components of $Y_{\sing}$. We assume that each $D_i$ is geometrically connected.

\begin{definition}
\label{def:d-semistable}
A scheme $Z$ over $k$ is $d$-semistable if there is an isomorphism $$\Ext_{\mathcal{O_Z}}^1 (\Omega_{Z/k}, \mathcal O_Z)\simeq \mathcal O_{Z_{\sing}}.$$
\end{definition}
\begin{definition}
By a \emph{logarithmic point} we mean the  scheme $\spec{k}$ with the logarithmic structure induced by the morphism:
\begin{equation}
\label{deflogpoint}
\N^m\To k; \quad e_i\mapsto 0;
\end{equation}
where $e_i$ stands for the canonical generator of $\N^m$. Here $m$ is the number of geometrically connected components of $Y_{\sing}$.
\end{definition}

Note that for every $1\leq i\leq m$ we have a logarithmic structure on $$\spec{k[x_0,\ldots, x_d]/ (x_0\cdots x_r)},$$ which is the one associated to the pre-logarithmic structure given by
the map $$\N^{i-1}\oplus \N^{r+1}\oplus \N^{m-i} \To \spec{k[x_0,\ldots, x_d]/ (x_0\cdots x_r)}$$ such that for the basic elements $e_j\in \N^{m+r}$:
\begin{equation}
e_j\mapsto
\begin{cases}
 0 & e_j\in \N^{i-1}\\
 x_j & e_j \in \N^{r+1}\\
 0 & e_j\in \N^{m-i}.
\end{cases}
\end{equation}
Note that this logarithmic structure commutes with the logarithmic structure over $\spec{k}^{\log}$ since we have a commutative diagram:
\[\xymatrix{
\N^{i-1}\oplus \N\oplus \N^{m-i} \ar[r]\ar[d] & k\ar[d] \\
\N^{i-1}\oplus \N^{r+1}\oplus \N^{m-i}\ar[r] & \spec{k[x_0,\ldots, x_d]/ (x_0\cdots x_r)}
}\]
where the upper horizontal morphism sends $e_j\mapsto 0$ for $e_j\in \N^{m}$ and the left vertical map is $id\oplus\text{diagonal}\oplus id$.
Let $Y$ be a proper $d$-semistable normal crossing variety over $k$. We endow $Y$ with the logarithmic structure given by:
\begin{enumerate}\label{logstructurebarX}
\item \'Etale locally on the neighborhood of a smooth point of $Y$, the logarithmic structure is given by the pull back of the logarithmic structure of $\spec{k}^{\log}$.
\item \'Etale locally  on the neighborhood of a point of $D_i$, the logarithmic structure is the pull back of the logarithmic structure of $\spec{k[x_0,\ldots, x_d]/ (x_0\cdots x_r)}$ described above.
\end{enumerate}
\begin{definition}
 We denote $Y^{log}/ \spec{k}^{log}$ the logarithmic scheme described above and we call it a \emph{normal crossing logarithmic variety} (NCL).

 We say that the NCL variety $ Y^{log}/ \spec{k}^{log}$ is \emph{simple} if the underlying scheme $Y$ is a simple normal crossing variety, where simple means that all its irreducible components are smooth and geometrically irreducible (SNCL).
\end{definition}

Now we follow \cite{Nakkajima:2000tl} and  \cite[ch 8]{Kato:1996wo}

Let $R$ be a fixed complete noetherian local ring with maximal ideal $\fk m$ and residue field $k$. We are mainly interested in the case $R= W(k)$.
 Let $Q$ be a fine and saturated (fs) non torsion monoid.  Let $R[[Q]]$ be the completion of the monoid ring $R[Q]$ with respect to the maximal ideal $\fk m + R[Q\setminus \set{1}]$. If the monoid is $\N$, then $R[[Q]]$ is isomorphic to $R[[t]]$ as a local $R$-algebra.

 Let $C_{R[[Q]]}$ be the category Artinian local $R[[Q]]$-algebras with residue field $k$, and $\wh{C}_{R[[Q]]}$ be the category of pro-objets of $C_{R[[Q]]}$.

 \begin{definition}
 For an object $A$ of $C_{R[[Q]]}$, we endow $\spec{A}$ with a logarithmic structure whose chart is $Q\to A$.  We denote this logarithmic scheme by $\spec{A}^{log}$. This data is equivalent to the following: $A$ is a $R$-algebra and there is a global chart $$\spec{A}^{log}\To (\spec{\Z[[Q]]},Q).$$
 \end{definition}
 Let $\beta:\spec{k}^{log} \to (\spec{\Z[[Q]]},Q)$ be a morphism of logarithmic scheme induced by a morphism $$Q\setminus \set 0 \to k; \quad q\mapsto 0.$$

Let $Y^{log}$ be a (fs) logarithmic scheme that is logarithmic smooth and integral over $\spec{k}^{log}$.

 \begin{definition}
 An fs logarithmic scheme $Y^{log}_A$ over $\spec{A}^{log}$ is called a a charted deformation of $Y^{log}$ over $\spec{K}^{log}$, if $Y^{log}_A$ is a logarithmic smooth scheme over $\spec{A}^{log}$ and \[Y^{log} \simeq Y^{log}_A \times_{\spec{A}^{log}} \spec{k}^{log}\] in the category of the fs logarithmic schemes.
 \end{definition}

 We have that $Y^{log}_A$ is automatically integral over $\spec{A}^{log}$.

 The charted deformations of $Y^{log}/\spec{A}^{log}$ define a functor
 \[D_{(Y^{log}, \beta)} \To (\text{\bf{Sets}}),\] Where $(\text{\bf{Sets}})$ is the category of sets.  Then we have the following:

 \begin{proposition}[F. Kato]\cite[theorem~8.7]{Kato:1996wo}.
 \label{prop:2.1}
 If $Y$ is proper, then the functor $D_{(Y^{log}, \beta)} $ has a hull.
 \end{proposition}

\begin{remark}
Returning to our situation of interest when we have a semistable model of a $K3$ surface $X_K$ over a local field $K$ (with residue field $k$), that is a diagram:
 \[
 \xymatrix{
 X_K\ar[r]\ar[d]  & X\ar[d] \\
 \spec{K\ar[r]} & \spec{\mathcal O_K}
 }\]
 with special fiber ${\oln X} = X\otimes k$. We set $Y={\oln X}$. Since $Y$ has a smoothing, that is, $Y$ lifts to a smooth $K3$-surface, then it is  $d$-semistable~\cite{Friedman:1983er} and~\cite{Olsson:2004tk}  then we can endow it with  the logarithmic structure studied in this section.  We can take $R= W:=W(k)$ as the ring of Witt vectors with coefficients in $k$.  Then proposition \ref{prop:2.1} is telling us that the deformation functor of the special fiber has a hull.
	
\end{remark}

  \begin{definition}
  Let $X^{log}/k^{log}$ be a NCL variety of pure dimension $2$. We say that $X^{log}/k^{log}$ is a \emph{normal crossing logarithmic $K3$-surface} if the underlying scheme $X$ is a proper scheme over $\spec{k}$ and
  \begin{enumerate}
  \item $H^1(X,\mathcal O_X) = 0$
  \item $\omega^2_{X/k}\simeq \mathcal O_X$.
  \end{enumerate}
  \end{definition}

  \begin{definition}\label{def:combinatorial} \cite[definition 3.2]{Nakkajima:2000tl}
  Let $X$ be a proper surface over a field $k$. Let $\oln k$ be an algebraic closure of $k$. $X$ is a combinatorial $K3$ surface if it satisfies one of the following conditions:
  \begin{enumerate}
  \item[I.] $X$ is a smooth $K3$ surface over $k$.
  \item[II.] $X\otimes_k \oln k = X_1\cup X_2\cup \cdots \cup X_N$ is a chain of smooth surfaces with $X_1$ and $X_N$ rational and the other elliptic ruled and two double curves on each of them are rulings. The dual graph of $X\otimes_K \oln k$ is a segment with end points $X_1$ and $X_2$.
 \item[III.] $X\otimes_K \oln k= X_1\cup X_2\cup \cdots \cup X_N$ is a chain of smooth surfaces and every $X_i$ is rational, and the double curves on $X_i$ are rational and form a cycle on $X_i$.
  \end{enumerate}
  and  $X$ has a logarithmic structure whose charts are given by its local normal crossing components   and $\omega^2_{X/k}\simeq \mathcal O_X$. 
  
 We say that $X$ is combinatorial of \emph{Type I}, \emph{Type II} or\emph{ Type II}  according to this definition. 
  \end{definition}
  
  Under this conditions~\cite[section 3]{Nakkajima:2000tl} proves that $H^1(X,\mathcal O_X)=0$. He also proves the following theorem~\eqref{thm:Nakka_type} and proposition~\eqref{prop:SNCLCombinatorial} that tell us that SNCL $K3$-surfaces are indeed a reasonable generalization of $K3$-surfaces and also gives us an equivalence between SNCL $K3$-surfaces and combinatorial $K3$-surface after extending scalars to the algebraic closure of $k$:

  \begin{theorem}\cite[theorem~3.3]{Nakkajima:2000tl}\label{thm:Nakka_type}
  Let $X$ be a combinatorial Type II or Type III $K3$ surface over a field $k$. Then $\Gamma(X,\omega^1_{X/k}) = 0$.
  \end{theorem}

  \begin{proposition}\cite[proposition~3.4]{Nakkajima:2000tl}
	  \label{prop:SNCLCombinatorial}
  Let $X^{log}/\spec{k}^{log}$ be  SNCL $K3$ surface. Then $X\otimes_k \oln k$ is a combinatorial $K3$ surface.
  \end{proposition}

 \begin{definition}
 We say that a SNCL or a combinatorial $K3$ surface is of type (I, II or III) if $X$ is of the respective type.
 \end{definition}

 \begin{theorem} \label{thm:p-adic famillyNakkajima}
 Let $k$ be an algebraically closed field of characteristic $p>0$. Let $X^{log}$ be a {projective} SNCL  $K3$ surface over $\spec{k}^{log}$. Then there exists a logarithmic smooth family $\fk X^{log}$ over $\spec{W[[u_1,\ldots, u_m]]}^{log}$ which is a charter deformation of $X^{log}$ (Automatically $\omega^2_{\fk X/W[[u_1,\ldots, u_m]]}$ is trivial). Where $m$ is the number of geometrically connected components of $X_{sing}$.  Moreover, the deformation functor  has the information of the deformations of the logarithmic structure associated to the irreducible components $X_1,\ldots, X_N$ of $\oln X$, in such a way that there exist closed subschemes $\fk X_1,\ldots,\fk X_N$, deformations of $X_1,\ldots X_N$ respectively,  and the logarithmic structure on $\fk X$ is the one associated with $\fk X_1,\ldots,\fk X_N$ as on page \pageref{logstructurebarX}.
 \end{theorem}

\begin{proof}
	If $\oln X$ is smooth, that is of type I, then this is the result of Deligne~\cite[ corollary 1.8]{Deligne:1981aa}. If $\oln X^{\log}$ is of type II or type III, then it is~\cite[proposition 5.9 and proposition 6.8]{Nakkajima:2000tl}.
\end{proof}

Nakkajima also gives the following corollaries:

 \begin{corollary}\cite[corollary~6.9]{Nakkajima:2000tl}
	 \label{cor:nakkajimasmoothing}
 Let $X$ be a projective $SNCL$ $K3$ surface over $k$. The following holds:
 \begin{itemize}
 \item There exists a projective semistable family $\mathcal Y$ over $\spec{W}$ whose special fiber is $X$.
 \item There exists a projective semistable family $\mathcal Y$ over $\spec{k[[t]]}$ whose special fiber is $X$.
 \end{itemize}
 \end{corollary}
 \begin{corollary}\cite[corollary~6.11]{Nakkajima:2000tl}
 \begin{itemize}
 \item Let $K_0$ be the fraction field of $W$ (resp. $k[[t]]$). The generic fiber $X_{K_0}$ of $\mathcal Y$ is a smooth $K3$ surface.
 \item Let $k$ be a field of characteristic $p>0$ and let $X^{log}$ be a projective SNCL $K3$ surface over $\spec{k}^{log}$. Then $\dim_k H^1(X, \omega_{X/k}^1) =20$.
 \end{itemize}
 \end{corollary}

In the argument for the proof, \cite{Nakkajima:2000tl} considers the family $$\fk X\to \spec{W[[u_1,\ldots, u_m]]}$$ and  specializes $W[[u_1,\ldots, u_m]]\to W$ by sending $u_i\to p$  getting the desired $\mathcal Y\to \spec{W}$. Similarly  he considers the map $$W[[u_1,\ldots, u_m]]\to W[[t]]\to W[[t]]/p =k[[t]]$$ and sends $u_i\to t$ and then reduces modulo $p$ to get the $\mathcal Y\to \spec{k[[t]]}$.

 These results are  Nakkajima's generalization, for the logarithmic case, of the results of Deligne~\cite{Deligne:1981aa} and Friedman~\cite{Friedman:1983er}.

\section{One Parameter Deformation of a Semistable $K3$-Surface} 
\label{sub:one_parameter_degeneration_of_k3_surfaces}

\begin{definition}\label{def:sncl-reduction}
	We say that a $K3$-surface $X_K$ over $K$, has \emph{combinatorial reduction}, if $X_K$ has a flat, proper and projective semistable model $X\To \mathcal O_K$ with special fiber $\oln X$ a SNCL $K3$-surface. 
\end{definition}

Note that if $X_K$ has combinatorial reduction, and if $\oln X$ is the special fiber of its semistable model, then $\oln X$ is combinatorial in the sense of definition~\eqref{def:combinatorial}.

The main result of the section is the following theorem:

\begin{theorem}\label{thm:liting of semistable K3}
Let $p>3$ be a prime number  and consider  $K$ be a finite extension of $K_0=W(k)[1/p]$ with $k$ algebraically closed.  Let $X_K\to\spec{K}$ be a $K3$-surface with combinatorial reduction (definition~\eqref{def:sncl-reduction} above). Then there exists a deformation $\mathcal X \to S:=\spec{W[[t]]}$ of  $\oln X$ satisfying the following: 
\begin{itemize}
    \item We let $0$ be the point of $S\ot_W K_0$, corresponding to the maximal ideal: $$t(W[[t]]\ot_W K_0)\subset W[[t]]\ot_W K_0,$$ then the fiber at $0$, $(\mathcal X\ot_W K_0)_0$, is a combinatorial $K3$-surface over $K_0$, of the same type of $\oln X$.
    \item For every point $x\in S\ot_W K_0$, with $x\ne 0$, then $(\mathcal X\ot_W K_0)_x$ is a smooth $K3$-surface over $k(x)$.
\end{itemize}
\end{theorem}

\begin{proof}
By theorem~\eqref{thm:p-adic famillyNakkajima} there exists a logarithmic smooth charter deformation of $\oln X $: $$\fk X^{\log}\To \fk S:=\spec{W[[u_1,u_2,\ldots, u_m]]}^{\log},$$ where $m$ is the number of geometrically connected components of $X_{sing}$ (so the deformation keeps information about the singularities), and that contain logarithmic deformations $\fk X_1, \fk X_2, \ldots, \fk X_N $ of the components of $\oln X= \oln X_1 \cup \oln X_2\cup\cdots\cup \oln X_N$. Moreover the logarithmic structure on $\fk S$ is the one given by the components $\fk X_i$ as on page \pageref{logstructurebarX} and proof by Kajiwara of \cite[thm. 6.8]{Nakkajima:2000tl}. In particular we have a lifting of each $\oln X_i$ to characteristic zero inside the corresponding $\fk X_i$. On the other hand by the work of [theorem 4.10]\cite{MR0422270} we know that $\oln X$ has an equisingular lifting to characteristic zero, and since $\fk S$ is obtained from the universal family of deformation of $\oln X$ (\cite[remak 3.6, prop.3.7]{Nakkajima:2000tl}), the equisingular lifting of $\oln X$ sits inside $\fk S$.

Let $\fk S \ot_W K_0$ be the scheme $\spec{W[[u_1,u_2,\ldots,u_m]]\ot_W K_0}^{\log}$ and let:
\[
(\fk S \ot_W K_0)_{sing}:=\set{x\in (\fk S\ot_W K_0) | (\fk X\ot_W K_0)_x\  \text{ is singular}}.
\]
Denote by $\fk S_{sing}$ the Zariski closure of $(\fk S\ot_W K_0)_{sing}$ in $\fk S$. 

Since being singular is a closed condition and $(\fk S\ot_W K_0)_{sing} \subset \fk S \ot_W K_0$ is a proper contention, we have that $\fk S_{sing}\subset \fk S$ is a proper closed immersion and therefore $$0\leq \dim \fk S_{sing}\leq \dim \fk S-1.$$
Let $x_0$ be the closed point of $(\fk S\ot _W K_0)_{sing}$ realizing the equisingular lifting to characteristic zero of $\oln X$, that is,  $(\mathcal X \ot_W K_0)_{x_0}$ is a combinatorial $K3$-surface over $K_0$ of the same type of $\oln X$, and let $y_0$ be a closed point of $\fk S_{sing}$ extending $x_0$.

Let $C$ be a smooth curve in $\fk S$ containing $y_0$ and normal to $\fk S_{sing}$. Let
$$ \widehat{\mathcal O}_{C,y_0} \simeq W[[t]]$$ denote the completion of the local ring of $C$ at $y_0$ with respect to the maximal ideal $\fk m$ of $y_0$. So we have a natural morphism $S:=\spec{W[[t]]}\to \fk S$. Let $\mathcal X\to S$ be the pull back of $\fk X\to \fk S$ with respect to $S\to\fk S$. Then $\mathcal X\to S$ satisfies the desired properties.
\end{proof}

 \begin{remark}
  If $X$ is the minimal semistable model for $X_K$ (which there exists for $p>3$~\cite{MR1269717,MR1689354}), then its special fiber $\oln X$ is automatically a SNCL $K3$-surface~\cite[section 4]{Maulik:2014aa} and~\cite{Nakkajima:2000tl}.
 \end{remark}



\section{Comparison Isomorphisms for Logarithmic K3-Surfaces} \label{cha:comparison_isomorphism_for_log_k3_surfaces}

In this section we recall the comparison isomorphism of~\cite{Andreatta:2012aa} and  for the special case in which $X_K$ is a smooth proper $K3$-surface over a $p$-adic  field $K$ (see remark~\ref{p-adicfield} on page~\pageref{p-adicfield}) and  $p>3$ is a fixed prime number.

As before we let $\mathcal O_K$ be the ring on integers of $K$ and we fix a uniformizer $\pi\in \mathcal O_K$. We also denote by $k= \mathcal O_K/\pi \mathcal O_K$ the residue field.

\subsection{$\bf p$-Adic Hodge Theory} \label{sec:_p_adic_Hodge_theory}

\subsubsection{Witt Vectors} \label{sub:witt_vectors}
Although we have already used Witt vectors in this paper, we would like to give a fast review of them, since it will be useful to recall the construction of the rings of periods. The standard reference for the proofs, properties and construction of Witt vectors is  the book~\cite[ch. 2 sec. 6]{SerreLF:1979}.

Let  $R$ be a perfect ring of characteristic $p$.

A \emph{strict $p$-ring with respect to $R$} is a ring $A$ (as always commutative and with one) such that $p$ is not nilpotent and $A$ is complete and separated with respect to the $p$-adic topology and with residue ring $A/pA = R$.

The  ring of Witt vectors with coefficients in $R$ is a strict $p$-ring with respect to $R$, and since $R$ is perfect, it is possible to construct at least one strict $p$-ring that in fact is unique, up to unique isomorphism. This ring is the ring of Witt vectors $W(R)$ over $R$.
Moreover by uniqueness  $W$ is functorial in $R$, that is, if we have a morphism $\phi:R\to S$ then it lifts to a map $W(\phi):W(R)\To W(S)$. In particular we have a lift to $W(R)$ of the Frobenius automorphism of $R$, which is also called the Frobenius automorphism.

For example, if $R= \F_p$, then $W(\F_p)=\Zp$. In general if $\F$ is a finite field, then $W(R)$ is the ring of integers of the unique unramified extension of $\Q_p$ whose residue field is $R$.  As a particular case, we have that if $K$ is a finite extension of $\Q_p$ and  $k= \OO_K/\pi\OO_K$ is its residue field, then $\Frac (W(k)) = K_0$ is the maximal unramified extension of $\Q_p$ in $K$. Another important example is when $R=\overline \F_p$ the algebraic closure of a finite field; in this case $W(R) = \mathcal O_{\widehat{\Q_p^{unr}}}$ the ring of integers of the maximal unramified extension of $\Qp$.

We want to understand the ring structure of $W(R)$. For $x=x_0\in R$ and for every $n$, choose a lifting ${\widetilde{x}_n}\in W(R)$ of $x^{p^{-n}}\in R$. The sequence $\set{{\widetilde{x}_n}}$ converges in $W(R)$. The \emph{Teichm\"uler map}  is the map: $[\ \,]:R\To W(R); \quad x\mapsto [x]:=\lim_n {\widetilde{x}_n}.$ The elements on the image of this map are called the \emph{Teichm\"uler elements}. It is multiplicative and a section of the natural projection. It turns out that the Teichm\"uler elements allows us to write any element $x\in W(R)$ in a unique way as
$x=\sum_{n=0}^\infty p^n[x_n], \text{ with } x_n\in R.$ Moreover, given two elements $x,y\in A$ we have that
\[
x+y=\sum_{n=0}^\infty p^n[S_n(x_n,y_n)] \quad\text{and}\quad xy=\sum_{n=0}^\infty p^n[P_n(x_n,y_n)]
\]
where $S_n, P_n$ are polynomials in $\Z[X_i^{p^{-n}}, Y_i^{p^{-n}}]$.

Given a $p$-adic field $K$ (see remark~\ref{p-adicfield} on page~\pageref{p-adicfield}), we denote by $\mathcal O_K$ its valuation ring, and we fix once and for all a uniformizer  $\pi \in \mathcal O_K$, so that its residue field is $k=\mathcal O_K/\pi\mathcal O_K$.  Denote by $\C_K= \widehat{\overline K}$; in case $k\subset \overline \F_p$ we have that $\C_K=\Cp$, the field of complex $p$-adic  numbers. We fix once and for all an algebraic closure $\overline K$ of $K$ and we denote by $G_K:=\Gal(\overline K/K)$ its absolute Galois group.

Let $\mu_{p^\infty} = \varprojlim_{n}\mu_{p^n}$ where $\mu_{p^n}:=\set{x\in \overline K : x^{p^n} = 1}$ with morphisms for every pair $n>m$:
\(
    \phi_{m,n}:\mu_{p^m} \to \mu_{p^n} \text{ sending } x\mapsto  x^{p^{n-m}}.
\)
Fix a primitive element $\xi \in \mu_{p^{\infty}}$\label{primitivelement} that is a sequence of primitive elements  $\xi = (1, \xi^{(1)}, \ldots, \xi^{(n)},\ldots )$ such that $(\xi^{(n+1)})^{p} = \xi^{(n)}$.

We have the following chain of fields:
\[
    K_0\subset K\subset K_n:=K(\mu_{p^n})\subset K_{\infty}:= K(\mu_{p^{\infty}}) \subset \overline K = \overline K_0\subset \C_K.
\]

If we denote by $\chi:G_K\To \Zp^*$ the cyclotomic character of $G_K$, that is the homomorphism of groups defined  by $\chi(\sigma)=\xi^{\chi(\sigma)}$ for every $\sigma\in G_K$, we have that the kernel of $\chi$ is exactly $H_K:=\Gal(\overline K/ K_{\infty})$ and therefore $\chi$ identifies $\Gamma_K:=\Gal(K_{\infty}/K)= G_K/H_K$ with the image of $\chi$ which is an open subgroup of $\Zp^*$.

Denote by $\OO:=W[[Z]]$ the polynomial ring with coefficients in $W=W(k)$ and variable $Z$. Consider the $W$-algebra homomorphism \(
    \OO\To\OO_K;\text{ sending } Z\mapsto \pi.
\) Finally denote by $P_\pi(Z)$  the minimal polynomial of $\pi$ in $\OO$.

\subsubsection{$\bf p$-Adic Representations} \label{sec:_p_adic_representations}

Now we recall some basics on the theory of $p$-adic representations, the main reference is~\cite{10017177619} in~\cite{periodespadiques}. Most of the results presented in this section can be found also in~\cite{MR2023292}.

\begin{definition}\label{def:padicrepresentation}
A $p$-adic  representation of $G_K$ of dimension $d$ is a continuous group homomorphism $\rho:G_K\To GL(V)$ for a finite dimensional (of dimension $d$) $\Q_p$-vector space $V$. Equivalently, it is a $\Qp$-vector space of dimension $d$ together with a continuous linear action of $G_K$.
\end{definition}

The collection of $p$-adic  representations form a category whose morphisms are given by $\Q_p$-linear and equivariant $G_K$-maps.  We denote by $\Rep_{\Q_p}(G_K)$  the category of $p$-adic  representations.

An important family of $p$-adic  representations of dimension one are the so called Tate twists of $\Q_p$. Precisely, let $r\in \Z$ and define $\Q_p(r)$ to be the one dimensional $\Q_p$-vector space  $\Q_p e_r$ with action of $G_K$ given by twisting by the $r$-power of the cyclotomic character, that is $\sigma(e_r)=\chi(\sigma)^r e_r$ for every $\sigma\in G_K$; this is called the $r$-th Tate twist of $\Q_p$.
If $V$ is a $p$-adic representation, we can construct a new $p$-adic representation by twisting $V$, that is, we let \(
    V(r):=V\otimes_{\Q_p} \Q_p(r);
\) which is again a $p$-adic  representation of dimension $\dim V$. If $X_K=A_K$ is an abelian variety over $K$, then the Tate module  $V_p:=T_pA\otimes_{\Zp} \Q_p$ is a $p$-adic  representation of dimension $d=2(\dim A)$. 
If $X_K$ is a $K3$-surface over $K$, then $V=\etl 2{X_{\overline K}}{\Q_p}$ is a $p$-adic representation of dimension $22$.

In order to study $p$-adic representations, Fontaine et al. constructed certain rings that are known as \emph{rings of periods}. They are topological $\Q_p$-algebras $B$, together with an action of $G_K$ and depending on $B$, some additional structures like filtrations, Frobenius, monodromy operator, etc. He also observed that the $B^{G_K}$-modules $D_B(B)$ defined as
\(
    D_B(V):=(B\ot_{\Q_p} V)^{G_K}
\) reveal important properties of the $p$-adic  representation $V$.

The $\Q_p$-algebra $B$ is $G_K$-regular if for any $b\in B$ such that the line $\Q_p b$ is $G_K$-stable, we have that $b\in B^*$. Note that if $B$ is $G_K$ regular, then for every $b\ne 0$ in $G^{G_K}$  the line $\Q_p b$ is $G_K$-stable, therefore for every $b\in B^{G_K}$, $b\in B^*$ and since $b^{-1}$ is also in $B^{G_K}$ we have that $B^{G_K}$ is a field. If $B$ is $G_K$ regular, then we have that $\dim_{B^{G_K}} D_B(V) \leq \dim_{\Q_p} V$.
\begin{definition}\label{def:Badmissible}
A $p$-adic  representation $V$ is $B$-admissible if $\dim_{B^{G_K}} D_B(V) = \dim_{\Q_p} V.$
\end{definition}

\subsubsection{The Ring of Periods $B_\dR$ and De Rham Representations} \label{sub:the_ring_b_dr_}
Let $R$ be the set of sequences $x=(x^{(0)},x^{(1)},\ldots, x^{(n)},\ldots)$ of elements in $\mathcal{O}_{\C_K}$  such that $(x^{(n+1)})^p=x^{(n)}$. We endow $R$ with a structure of  a ring with  product $*$ and sum $+$ laws defined as $x*y=(x^{(n)}y^{(n)})_{n\in\N}\text{ and } x+y=(s^{({n})})_{n\in\N}$
where $s^{(n)}=\lim_{m\to\infty}(x^{(n+m)}+y^{(n+m)})^{p^m}$
which converge in $\mathcal{O}_{\C_K}$.
With these operations  $R$ is a commutative domain whose unit element is $1=(1,1,\ldots,1)$. This ring is usually denoted by $\bf{\widetilde E}^+=\bf{\widetilde E}^+_{\C_K}$.
Also note that $p* 1=\lim_{m\to +\infty}{\underbrace{(1+\ldots + 1)}_{p-\text{times}} }^{p^m}=0$, thus $R$ is of characteristic  $p$. The Frobenius $x=(x^{(n)})\mapsto x^p=((x^{(n)})^p)$ on $R$ is  an isomorphism, and so, $R$ is perfect ring.

Even more, we have a natural action of $\Gal(\overline K/K)$ on $R$ trough its action on $\OO_{\C_K}$ and a valuation defined as $\text{val}(x)=\text{val}(x^{(0)}).$
With the topology induced by the valuation, $R$ is separated and complete with a residue field $R/\{x|\text{val}(x)>0\}\cong\overline k$.

Since $R=\bf{\widetilde E}^+$ is a perfect ring we can consider the Witt vectors $A_{\inf}:=W(R)$ with coefficients in $R$. Every element of $A_{\inf}$ can be written in a unique way as the sum $\sum_{n=0}^{+\infty}p^n[x_n]$
where $x_n\in R$ and $[x_n]$ is its multiplicative representative or Teichm\"uler  representative in $W(R)=A_{\inf}$.
We have a surjection
$$\theta: A_{\inf}\to \OO_{\C_K}; \qquad \sum_{n=0}^{+\infty}p^n[x_n]\mapsto \sum_{n=0}^{+\infty}p^nx_n^{(0)}.$$
Remark that $\theta([\overline \pi])=\pi$ and $\theta([\overline p])=p$ and that  $\ker{\theta}$ is a principal ideal generated by $p-[\overline{p}]$, where $\overline{p}:=(p^{(n)})\in R$ is such that $p^{(0)}=p$, also  $\overline\pi \in R$ such that if $\overline\pi=(\pi^{(n)})$ then $\pi^{(0)}=\pi$. Note that the primitive element (see page\pageref{primitivelement}) $\xi$ is in $R$ and that $\theta(1-[\xi]) = 0$.
The ring $A_{\inf}$ is complete for the topology defined by the ideal $(p,\ker(\theta))=(p,[\underline{p}]).$

\begin{definition}
  The ring $B_{\dR}^+$ is the completion of $A_{\inf}[1/p]$ with respect to the ideal $\ker(\theta)=(p-[\overline{p}])$.
\end{definition}
We extend the surjection $\theta:A_{\inf}\to\OO_{\C_K}$ to $\theta:B_{\dR}^+\to\C_K$.
We have that $B_{\dR}^+$ is a complete ring with a discrete valuation  and  maximal ideal $\ker\theta=(p-[\overline{p}])B_{\dR}^+$ and residue field
$B_{\dR}^+/\ker(\theta)\cong \C_K.$

We can consider several topologies in $B_{\dR}^+$.
  We endow $B_{\dR}^+$ with the topology so that $p^mW(R)+(\ker\theta)^k$
  forms a base of neighborhoods of $0$, where  $(m,k)\in\N^2$. $B_{\dR}^+$ is complete and separated for this topology.

  There exists a natural and continuous action of $G_K$ in $B_{\dR}^+$ through the action on $R$ and this action commutes with $\theta$.

  $\overline\Q_p$ is identified canonically with  the algebraic closure of $\Q_p$ in $B_{\dR}^+$ and the following diagram commutes:
  
\[
\xymatrix{
\overline{\Q}_p \ar[r] \ar@{=}[d] & B_{\dR}^+ \ar[d]^{\theta} \\
\overline{\Q}_p\ar[r] & \mathbb C_{K}.}
\]

In fact, in the case where we give $\overline\Q_p$ the topology induced by $B_{\dR}^+$ (which is not $p$-adic), Colmez proved that $B_{\dR}^+$ is the completion of $\overline\Q_p$ for this topology, and thus $\overline\Q_p$ is dense in $B_{\dR}^+$.
\begin{definition}
    \label{def:BdR} We define $B_{\dR}$ as the fraction field of $B_{\dR}^+$, that is $B_{\dR}:=\Frac(B_{\dR}^+).$
\end{definition}
We extend naturally $\theta$ to $B_{\dR}$ and we give to it a filtration, defined as $$\Fil^i{B_{\dR}}:= (\ker(\theta))^i.$$  Remark that if  $x  \in \Fil^1(B_{\dR}) = \ker(\theta)$, is non zero,  then $B_{\dR}=B_{\dR}^+[x^{-1}]$.

Since  $\theta(1-[\xi])=0$ the element $1-[\xi]$ is small with respect to the topology on $B^+_{\dR}$ and the logarithm of this element converges in $B^+_{\dR}$, that is, there exists an element $t\in B^+_{\dR}$  such  that
\[
    t=\log([\xi]):=-\sum_{n=1}^{\infty}\frac{(1-[\xi])^n}{n}.
\]
If $\sigma\in G_K$ then $\sigma*t = \sigma(\log([\xi])) = \log([\xi^{\chi(\sigma)}]) =\chi(\sigma)t.$ Moreover since $t\in Fil^1 (B_{\dR})$ we also have that $B_{\dR}= B_{\dR}^+[1/t]$ and the filtration is such that $\Fil^i B_{\dR}= t^iB_{\dR}^+$.
The field $B_{\dR}$ satisfies that $B_{\dR}^{G_K} = K$.
We say that a $p$-adic  representation $V$ is de Rham if $V$ is $B_{\dR}$-admissible in the sense of definition~\ref{def:Badmissible}.

\subsubsection{The Ring of Periods $B_\cris$ and Crystalline Representations} \label{ssub:construction_of_b_cris_}

$A_{\cris}$ is the $p$-adic  completion of the divided power envelope of $A_{\inf}$ with respect to the ideal generated by $p$ and $\ker(\theta)$. We endow $A_{\cris}$ with the $p$-adic  topology and the divided power filtration. 

 \begin{definition}\label{def:Bcris}
 We define $B_{\cris}$ as the ring:
 \(
     B_{\cris}:=A_{\cris}[1/t]
 \)
 with the inductive limit topology and filtration given by:
 \[
     \Fil^i B_{\cris}:=\sum_{m\in\N}t^{-m}\Fil^{m+i}A_{\cris}.
 \]
 \end{definition}

We have that $B_{\cris}^{G_K}=K_0$. $B_{\cris}$ has a Frobenius $\phi$ compatible with the Frobenius of $W$ and is such that $\phi(t)=pt$. We have that $B_{\cris}$ is an algebra over $K_0$ which is a subring of $B_{\dR}$ and $G_K$-stable.

A $p$-adic representation $V$ is \emph{crystalline}, if it is  $B_{\cris}$-admissible as in definition~\ref{def:Badmissible}.

\subsubsection{The Ring of Periods $B_\st$ and Semistable Representations}  \label{sub:b_st}
 \begin{definition}\label{def:Bst}
We define $B_{\st}$ as the ring of polynomials $B_{\cris}[Y]$ on the variable $Y$ such that:
 \begin{itemize}
     \item We extend the Frobenius $\phi$ of $B_{\cris}$ to $B_\st$ by letting $\phi(Y)=Y^p$.
     \item We extend the action of $G_K$ on $B_{\cris}$ by
 \[
     \sigma*Y= Y+c(\sigma)t:\quad\text{ for } \sigma\in G_K
 \]
 where $c(\sigma)$ is defined by the formula $\sigma(p^{1/p^n})=p^{1/p^n}(\xi^{(n)})^{c(\sigma)}$.
 \item We define a \emph{Monodromy} operator on it as $N_{st}:=-d/dY$.
 \end{itemize}
 \end{definition}
 
 A $p$-adic  representation $V$ is \emph{semistable} if it is $B_{\st}$-admissible.

 We have that $B_\st$ is a $K_0$-algebra with an action of $G_K$ and containing $B_{\cris}$. Moreover $B_{\st}^{G_K}=K_0$ and  $B_{\st}^{N_{st}=0}=B_{\cris}$.
 
 \begin{remark}\label{rem:st-cris}
	In particular note that a semistable representation $V$ is crystalline it the monodromy $N_{st}=0$.  
 \end{remark}

\subsubsection{The Ring of Periods $B_{\log}$} \label{ssub:Blog}
For this constructions the main reference is~\cite[sec 3]{Kato:1994we}.

We denoted by $\OO=W[[Z]]$ for $W=W(k)$. Let $\OO_{\cris}$ the $p$-adic  completion of the divided power envelope of $\mathcal O$ with respect to the ideal $(p,P_{\pi}(Z))$, where $P_{\pi}(Z)$ is the minimal polynomial of $\pi$ with coefficients on $W$. We extend the Frobenius of $W$ to $\mathcal O$ by letting  it act on $Z$ as $Z\mapsto Z^p$ and the usual Frobenius on $W$. Finally let $\omega^1_{\cris/W}\simeq \mathcal O_{\cris}\tfrac{dZ}{Z}$ be the continuous log 1-differential forms of $\mathcal O_{\cris}$ relative to $W$.

 Define $A_{\log}$ as the $p$-adic  completion of the log divided power envelope of the morphism ring $A_{\inf}\ot_{W} \OO$ with respect to the kernel of the morphism
 \[
  \theta\ot \theta_{\mathcal O}: A_{\inf}\ot_W \mathcal O \To \mathcal O_{\C_K}.
 \]

Consider the element $u:=\frac{[\overline \pi]}{Z}$. Then we have that $A_{\log}$ is isomorphic to the $p$-adic  completion  $A_{\cris}\set{\langle V\rangle}$ of the divided power polynomial ring over $A_{\cris}$ in the variable $V$ by a morphism:
\[
A_{\cris}\set{\langle V\rangle}\To A_{\log}; \quad V\mapsto \frac{[\overline \pi]}{Z}-1 = u-1.
\]
Then $A_{\log}\simeq A_{\cris}\set{\langle u-1\rangle}$.

We endow $A_{\log}$ with the $p$-adic  topology and the divided power filtration.
\begin{definition}\label{def:Blog}
We define the ring $B_{\log}$ as the ring $B_{\log}:=A_{\log}[t^{-1}]$ with the inductive limit topology and filtration defined by
\[
    \Fil^n B_{\log}:=\sum_{m\in \N}\Fil^{n+m} A_{\log}t^{-m}.
\]
\end{definition}

We have a Frobenius on $A_{\log}$ that extends the Frobenius on $A_{\cris}$ by letting $u\mapsto u^p$ and we extend it to $B_{\log}$ by letting $t\mapsto pt$.

We have a continuous action on $B_{\log}$ of the group $G_K$ acting trivially on $W$ and on $\mathcal O$, and acting on $A_{\inf}$ through the action on $\mathcal O_{\C_K}$. Moreover we have a derivation on $B_{\log}$
\[
    d:B_{\log}\To B_{\log}\frac{dZ}{Z}
\]
which is $B_{\cris}$ linear and satisfies $d((u-1)^{[n]}) = (u-1)^{[n-1]}u\frac{dZ}{Z}$.

\begin{definition}\label{def:BlogMonodromy}
The Monodromy operator on $B_{\log}$ is the operator
\[
    N_{\log}:B_{\log}\To B_{\log}; \qquad\text{ such that }\quad d(f)=N_{\log}(f)\frac{dZ}{Z}.
\]
\end{definition}

We can recover the ring $B_{st}$ from $B_{\log}$ by considering the largest subring of $B_{\log}$ in which $N_{\log}$ acts as a nilpotent operator~\cite[theorem 3.7]{Kato:1994we}.

\begin{definition}
	A p-adic representation $V$ is $B_{\log}$-admissible if:
	\begin{itemize}
		\item  $D_{\log}(V):=\left(B_{\log} \otimes_{\Qp} V \right)^{G_K}$ is a free $B_{\log}^{G_K}$-module;
		\item the morphism $B_{\log} \otimes_{B_{\log}^{G_K}} D_{\log}(V) \To B_{\log}\otimes_{\Qp} V$ is an isomorphism, strictly compatible with filtrations.
	\end{itemize}
\end{definition}

In this case (see~\cite[sec 6.1]{Breuil:1997aa}) $M:=D_{\log}(V)$ is endowed with a monodromy operator $N_M$ compatible with $N_{\log}$ via  Leibniz rule, a decreasing exhaustive filtration $\{Fil^nM\}$ which satisfies Griffiths' transversality with respect to $N_M$ and such that the multiplication map $B_{\log}^{G_K} \times M\to M$ is compatible with the filtrations, a semilinear Frobenius morphism $\phi_M:M\to M$ such that $N_M\circ \phi_M = p \phi_M\circ N_M$ and with determinant invertible in $B_{\log}^{G_K}$. That is $M$ is a finite free filtered $(\phi, N)$-module over $B_{\log}^{G_K}$. 


\subsubsection{An Admissibility Criterion}  
\label{ssub:an_admissibility_criterion}
We recall the admissibility criterion of~\cite[2.1.1 page 140]{Andreatta:2012aa} which is very similar to the admissibility criteria on~\cite[theorem 4.3]{Colmez:2000vq}.

Let $M$  be a finite free $B_{\log}^{G_K}$-module, which is a finite $(\phi,N)$-module. The map $$B_{\log}\to B_{dR};\quad  Z\mapsto \pi$$ has image $\oln B_{\log}$.  We define
\[
    V_{\log}^0(M):=(B_{\log}\otimes_{B_{\log}} M)^{N=0,\, \phi=1}
\]
and
\[
    V_{\log}^1(M):=(B_{\log}\otimes_{\oln B_{\log}} M)/\Fil^0(\oln B_{log}\ot_{B_{\log}^{G_K}} M).
\]
Let $\delta(M):V_{\log}^0\To V_{\log}^1(M)$ be the map given by the composite of the inclusion and projection $$V_{\log}^0(M)\subset B_{\log}\otimes_{B_{\log}^{G_K}} M\To \oln B_{\log}\otimes_{B_{\log}^{G_K}}M.$$
We define  $V_{\log}(M):= \ker (\delta(M))$. Then

\begin{proposition} \cite[proposition  2.3]{Andreatta:2012aa}
	\label{prop:LogAdmissible}\hspace{0.1cm}
\begin{itemize}
    \item A filtered $(\phi,N)$-module $M$ over $B_{\log}^{G_K}$ is \emph{admissible} if and only if $V_{\log}(M)$ is a finite dimensional $\Q_p$-vector space and $\delta(M)$ is surjective.
    \item Moreover, if $M$ is admissible then $V:=V_{\log}(M)$ is a finite dimensional, semistable $G_K$-representation and $D_{\log}(V)=M$.
\end{itemize}
\end{proposition}


\subsection{The Comparison Isomorphisms} \label{ssub:the_comparison_isomorphisms}

We consider on $X$ the induced logarithmic structure given by its special fiber $\oln X$, which is a normal crossing divisor and we give to ${\oln X}$ the pull back logarithmic structure as in section~\eqref{sec:logarithmic_structures} denoted by $X^{\log}$ and $\oln X^{\log}$ as usual.

Let $S^{\log}:=\spec{W[[t]]}^{\log}$ where $W=W(k)$ and the logarithmic structure
 on $S$ is the induced by the pre logarithmic structure $\N\to W[[t]]; \  n\to t^n$. We have seen on theorem~\eqref{thm:liting of semistable K3} that the deformation $\mathcal{X}^{\log}\to  S^{\log}$ of the special fiber ${\oln X}$ may be chosen such that it has the properties as in theorem~\eqref{thm:liting of semistable K3}:

 \begin{itemize}
    \item  $(\mathcal X\ot_W K_0)_0$ is a combinatorial $K3$-surface over $K_0$ of the same type of $\oln X$.

    \item For every point $x\in S\ot_W K_0$, with $x\ne 0$, then $(\mathcal X\ot_W K_0)_x$ is a smooth $K3$-surface over $k(x)$.
\end{itemize}

 Remember that $\bb Y:= (\mathcal X\ot_W K_0)_0$ is of the same type of $\oln X$,  in particular ${\oln{X}}$ is smooth if and only if $\bb Y$ is smooth. 
  We consider on $\mathcal X_{K_0}=\mathcal X\ot_W K_0$ the logarithmic structure defined by the divisor with normal crossings $\bb Y\to \mathcal X_{K_0}$ and on $\bb Y$ the inverse image logarithmic structure.

 Denote by $D:=H^2_{\log-dR}(\bb Y)$ the logarithmic de Rham cohomology of $\bb Y/K_0$. Then $D$ has a natural structure of filtered, $(\phi, N)$-module over $K_0$ obtained by its identification with $H^2_{log-\cris}({\oln X}/W)[1/p]$ with the logarithmic crystalline cohomology of ${\oln X}$ over $W$. More precisely the structure of filtered $(\phi,N)$-module of $D$ can be explicitly described as follows:

 Let $\mathcal H = H^2_{log-\dR}(\mathcal X /S)$ denote the locally free $\OO_{S}$-module of relative logarithmic de Rham cohomology of $\mathcal X$ over $S$. It is endowed with a logarithmic integrable connection $\nabla$, the Gauss-Manin connection, and a Frobenius $\phi$. Moreover, $\mathcal H$ can be naturally identified with $H^2_{\log-\cris}({\oln X}/W[[t]])[1/p]$ therefore we have the identifications:

 \begin{itemize}
     \item $\mathcal H_0:=\mathcal H/t\mathcal H \simeq D$;

     \item $\mathcal H_0\ot_{K_0} K \simeq H^2_{\dR}(X_K)$.
 \end{itemize}

 Hence, we  have natural identifications $D_K:=D\otimes_{K_0} K \simeq H^2_{dR}(X_K)$ and so we define the filtration on $D_K$ to be the inverse image of the Hodge filtration on $H_{dR}^2(X_K)$.

 Moreover we define the $p$-adic  monodromy operator $N_p$ on $D$ to be the residue of $\nabla$ modulo $t\mathcal H$ and the Frobenius $\phi_0$ on $D$ to be the reduction modulo $t\mathcal H$.

 We also denote by $V:= H^2_{\et}(X_{\oln K},\Q_p) $; it is a $p$-adic  $G_K$-representation.

In \cite[sec. 2.4.9]{Andreatta:2012aa} is proved a more general version of the following theorem, that specialized to our case reads as follows:

 \begin{theorem}[Comparison Isomorphisms] \cite[theorem 2.33 and proposition 2.37]{Andreatta:2012aa}
	 \label{thm:comparisonIovitaAndreatta}
	 
 $V$ is a semistable $G_K$-representation and we have a natural isomorphism of filtered, $(\phi,N)$-modules: $D_{st}(V)\simeq D$.
 \end{theorem}

 For the proof, \cite{Andreatta:2012aa} considers $M:=D\ot_{K_0}(B_{\log}^{G_K})$ with its induced filtered $(\phi, N)$-module structure and proves that:
 \begin{itemize}
     \item[a)] $M$ is an admissible filtered $(\phi, N)$-module and
     \item[b)] $V$ and  $V_{\log}(M)$ are isomorphic as $G_K$-representations.
 \end{itemize}
 
 In order to prove $a$ and $b)$ \cite{Andreatta:2012aa} use the ``fundamental exact diagram'' of Fontaine Shaves on Faltings' site associated to the pair $(X, X_K)$~\cite[sec 2.3.9]{Andreatta:2012aa}.
  
   Proposition~\eqref{prop:LogAdmissible} (see also \cite[proposition 2.1]{Andreatta:2015aa}) now implies that:
  \begin{proposition}\label{prop:comparison2}
 $$D_{\log}(V):=M=D\otimes_{K_0} B_{\log}^{G_K}$$ and so $D\simeq D_{st}(V)$, as filtered, Frobenius, monodromy modules. In particular there is an identification $N_p=N_{st}$.
  \end{proposition}


\section{The Main Theorem} \label{sec:the_main_theorem_v2}

The following theorem  is an analogue of the~\cite{Kulikov:1977aa,Persson:1981wp,MR0466149}-classification theorem of the central fiber of a semistable degeneration of complex $K3$-surfaces in terms of the monodromy, but now over a $p$-adic  field.

The new part of the theorem is that we can distinguish the three possible types of the special fiber of a semistable $K3$-surface over a $p$-adic  field $K$, in terms of the ($p$-adic) monodromy operator $N_{st}$ on $D_{st}(H^2_{\et}(X_{\oln K},\Qp))$. As a consequence of this result, we get a criterion for the good reduction of the semistable $K3$-surface in terms of the $p$-adic representation $H^2_{\et}(X_{\oln K},\Qp)$ analogous to the~\cite[theorem II.4.7]{Coleman:1999aa} for abelian varieties and~\cite[theorem 1.6]{Andreatta:2015aa} for curves.

\begin{theorem}\label{thm:MainTheorem}
Let $X_K\to \spec{K}$ be a smooth projective $K3$-surface and let $X\to \spec{ \mathcal O_K}$ be a semistable minimal model of $X_K$ (or we can just assume that $X_K$ has combinatorial-reduction as in definition~\eqref{def:sncl-reduction}). Let ${\oln X}$ be the special fiber of $X$. 
Let $D_{st}=D_{st}( H^2_{\et}(X_{\oln K},\Qp))$ and $N_{st}:D_{st}\to D_{st}$
be the monodromy operator on $D_{st}$. Then we have $3$ possibilities for the special fiber ${\oln X}$, distinguished in terms of the nilpotency degree of the monodromy operator $N_{st}$, as follows:
\begin{itemize}
    \item[I.]  $N_{st}=0$ if and only if
     ${\oln X}$ is a nonsingular $K3$ surface.
    \item[II.]  $N_{st}\ne 0$ but $N^2_{st}=0$ if and only if
     ${\oln X}= \cup_{i=1}^n V_i$ where the $V_i$ are rational surfaces and $V_2, \ldots, V_{n-1}$ are elliptic ruled surfaces.
    \item[III.] $N_{st}^2\ne 0$ but $N_{st}^3=0$ if and only if
     ${\oln X}= \cup_{i=1}^n V_i$ where all the $V_i$ are rational surfaces.
\end{itemize}
That is, the type of the combinatorial special fiber $\oln X$ can be distinguish in terms of the monodromy operator $N_{st}$ in exactly the same way as in \cite{MR0466149} and \cite{Kulikov:1977aa} for the complex case.
\end{theorem}

\begin{proof}
 As $X\to\spec{\mathcal O_K}$ is a minimal semistable model of $X_K\to \spec{K}$ (see~\cite[corollary [3.7]{MR1397057}), the special fiber is a SNCL $K3$-surface~\cite[theorem 4.1]{Maulik:2014aa} and~\cite[section 2 and 3]{Nakkajima:2000tl} and therefore it is a combinatorial $K3$-surface by proposition~\eqref{prop:SNCLCombinatorial} (or just assuming that we have combinatorial reduction), i.e. it is of type I, II or III. So the remaining thing to prove is that we can distinguish these 3 cases in terms of the nilpotency degree of the monodromy operator $N_{st}$.

\textbf{Step 1}.
We consider on $X$ the induced logarithmic structure given by its special fiber $\oln X$, which is a normal crossing divisor and we give to ${\oln X}$ the pull back logarithmic structure as in section~\eqref{sec:logarithmic_structures} denoted by $X^{\log}$ and $\oln X^{\log}$ as usual.

Remember that by theorem~\eqref{thm:liting of semistable K3} there exists a deformation $$\mathcal X \to S:=\spec{W[[t]]}$$ of  $\oln X$ such that:
\begin{itemize}
    \item If we let $0$ denote the point of $S\ot_W K_0$ corresponding to the maximal ideal $t(W[[t]]\ot_W K_0)$. Then $\bb Y:=(\mathcal X\ot_W K_0)_0$ is a combinatorial $K3$-surface over $K_0$ of the same type of $\oln X$.

    \item For every point $x\in S\ot_W K_0$, with $x\ne 0$, then $(\mathcal X\ot_W K_0)_x$ is a smooth $K3$-surface over $k(x)$.
\end{itemize}

 We considered on $\mathcal X_{K_0}:=\mathcal X\ot_W K_0$ the logarithmic structure defined by the divisor with normal crossings $\bb Y\to \mathcal X\ot_W K_0$ and on $\bb Y$ the inverse image logarithmic structure.

 Denote by $D:=H^2_{\log-dR}(\bb Y)$ the logarithmic de Rham cohomology of $\bb Y/K_0$ where $K_0 = \Frac(W(k))$. Then $D$ has a natural structure of filtered, $(\phi, N)$-module over $K_0$ obtained by its identification of $H^2_{\log-\cris}({\oln X}/W)[1/p]$ with the logarithmic crystalline cohomology of ${\oln X}$ over $W$.

By proposition~\eqref{prop:comparison2}, the monodromy operator $N_{st}$ on $D_{st}(H^2_{\et}(X_{\oln K},\Qp))$ can be identified with the residue $N_p$ of the Gauss-Manin connection $\nabla$ modulo $t\mathcal H$,  that is $N_p$ is an endomorphism of $H^2_{\log-\dR}(\bb Y/S[1/p]))$.

\textbf{Step 2.} Now fix once and for all an embedding of $K_0\to \C$. Consider the base change of $\mathcal X_{K_0}:=\mathcal X\ot_W K_0$ with respect to the induced embedding $W[[t]]\ot_W K_0\to \C[[t]]$. We have a complex family $\mathcal X_\C:=\mathcal X_{K_0}\ot \C\To \spec{\C[[t]]}$ with special fiber $Y_\C=\bb Y\ot \C \to \spec{\C}$ a combinatorial $K3$-surface and generic fiber  a smooth $K3$-surface  $X_{\C((t))}$.  Let $\mathcal S=S[1/p]\ot\C=\spec{\C[[t]]}$.

We endow $\mathcal X_\C, \mathcal S, Y_{\C} $ with the usual logarithmic structures and, we denote them as  $\mathcal X_\C^{\log}, S^{\log}$,\\ $Y_{\C}^{\log}$ respectively.

Consider now the logarithmic de Rham cohomology $H^2_{\log-\dR}(\mathcal X_\C^{\log}/\mathcal{S}^{\log})$, it is a free $\mathcal O_{\mathcal S}$-module or rank $22$ with an integrable logarithmic connection (The logarithmic Gauss-Manin connection):
\[
\nabla: H^2_{\log-\dR}(\mathcal X_\C^{\log}/\mathcal S^{\log}))\To H^2_{\log-\dR}(\mathcal X_\C^{\log}/\mathcal S^{\log}))\otimes_{ \mathcal O_{\mathcal S}} \omega^1_{\mathcal S/\C}.
\]
The fiber of $ H^2_{\log-\dR}(\mathcal X_\C^{\log}/\mathcal S^{\log}))$ at the special point is $ H^2_{\log-\dR}(Y_{\C}^{\log})$ that is
\[
H^2_{\log-\dR}(Y_{\C}^{\log}) \simeq H^2_{\log-\dR}(\mathcal X_\C^{\log}/\mathcal S^{\log}))/ t H^2_{\log-\dR}(\mathcal X_\C^{\log}/\mathcal S^{\log})).
\]

We also have the operator $N_\C:=\Res_{t=0}\nabla$ which is a $\C$-linear, nilpotent operator on $H^2_{\log-\dR}(Y_{\C}^{\log})$.

Let us notice that the pair $\left( H^2_{\log-\dR}(Y_\C^{\log}/\C^{\log}), N_\C \right)$ is the base change to $\C$, via the embedding $K_0\subset \C$, of the pair $\left( H^2_{\log-\dR}(\bb Y^{\log}/K_0^{\log}), N_p\right)$.

\textbf{Step 3.} Now we use~\cite{Artin:1969wg}:

We associate to the family $\mathcal X_\C\to \mathcal S=\spec{\C[[t]]}$ above a family of $K3$-surfaces $\mathcal Y\to \Delta$, over the complex open unit disk $\Delta$. This family has the property that if we base change it over $\mathcal S$, we obtain a  family $\mathcal{X}'_{\C}$ which is congruent to $\mathcal{X}_\C$ modulo $t^m\C[[t]]$ for some large $m>1$. It follows that:
\begin{itemize}
    \item $\mathcal Y|_{\Delta-\set{0}}$ is a smooth projective family of $K3$-surfaces.
    \item The central fiber $\mathcal Y_0\simeq \mathcal{X}'_0 \simeq Y_{\C}^{an}$.
\end{itemize}
Here $Y_{\C}^{an}$ denotes the complex analytic variety associated to the complex points of $Y_{\C}$ (the usual GAGA functor).

Now we use the Monodromy criterion~\cite[pag. 112]{Morrison:1984uh} given by the Clemens-Schmidt exact sequence  to the family $\mathcal Y\to \Delta$  (this  criterion leads to the proof  of the classical~\citeauthor{Kulikov:1977aa,Persson:1981wp,MR0466149}-classification theorem, as we can see in ~\cite[pag. 113]{Morrison:1984uh}).

Consider the GAGA functor $an$ sending a complex algebraic $Z$ variety to its associated complex analytic variety $Z^{an}$. We let $N_{an}$ be the monodromy operator on $H^2_{\log-\dR}(Y_{\C}^{\log,an})$. By~\cite{deligne14004equations} it can be seen, up to non-zero constant, as the residue at zero of the Gauss-Manin connection:
{\footnotesize{\[
\nabla_{an}: H^2_{\log-\dR}(\mathcal Y/\Delta^{\log})\To H^2_{\log-\dR}(\mathcal Y/\Delta^{\log})\otimes \omega^1_{\mathcal Y/\Delta}.
\]}}
Therefore, $N_{an}$ can be seen (by the previous analysis) as the residue of the Gauss-Manin connection $\nabla'$ on $ H^2_{\log-\dR}((\mathcal{X}_{\C}')^{\log}/\mathcal S^{\log})$. But we have:
{\footnotesize{\begin{equation}\label{eq:dridentification}
    H^2_{\log-\dR}((\mathcal{X}_{\C}')^{\log}/\mathcal S^{\log})/ t^m H^2_{\log-\dR}((\mathcal {X}_{\C}')^{\log}/\mathcal S^{\log}) \simeq H^2_{\log-\dR}(\mathcal X_{\C}^{\log}/\mathcal S^{\log}))/t^m H^2_{\log-\dR}(\mathcal X_{\C}^{\log}/\mathcal S^{\log}))
\end{equation}}}
and $\nabla'\equiv \nabla \pmod{t^m \C[[t]]}$.

This implies that the residue of $\nabla_{an}$ $\nabla$ and $\nabla'$ are the same under the identification
\[
H^2_{\log-\dR}(\mathcal Y_0)\simeq H^2_{\log-\dR}(Y_{\C}^{\log,an}) \simeq H^2_{\log-\dR}((\mathcal {X}'_0)^{\log}).
\]
In other words $N_{an}=N_{\C}$, which is the base change to $\C$ of $N_p$.

Now we apply the description of $\mathcal Y_0$ in therms of $N_{an}$ for the family $\mathcal Y$ given by the Clemens-Schmidt exact sequence~\cite[pag. 113]{Morrison:1984uh}. So if $N_{an}=0$ then $\mathcal Y_0$ is of type I. If $N_{an}\ne 0$ but $N_{an}^2=0$ then $\mathcal Y_0$ is of type II and if $N_{an}^2\ne 0$ but $N_{an}^3=0$ then $\mathcal Y_0$ is of type III. Where the type I, type II and type III are as in the theorem~\eqref{thm:MainTheorem}.

Since $Y_{\C}^{an} =\mathcal Y_0$, then also $Y_{\C}$ is of the same type, and since $\bb Y\ot \C = Y_{\C}$ we have that also $\bb Y$ is of the same type, hence $\oln X$ is of the same type.

Moreover we have seen that  $N_\C=N_{an}=N_p=N_{st}$ up to constants. Which implies that we can distinguish the three possible types of $\oln X$ in terms of the nilpotency degree on $N_{st}$ as stated in the theorem.
\end{proof}

The following theorem, which is in fact a corollary of theorem~\eqref{thm:MainTheorem}, is a generalization of~\cite[theorem II.4.7]{Coleman:1999aa} for abelian varieties. 

\begin{theorem}[Corollary]\label{thm:MainTheoremCrystalline}
Let $X_K\to \spec{K}$ be a semistable $K3$-surface over the $p$-adic  field $K$ with minimal semistable integral model $X\to\spec{\mathcal O_K}$ and with projective special fiber $\oln X$ over the algebraic closed field $k=\mathcal O_K/\pi \mathcal O_K$. Let $V:= H^2_{\et}(X_{\oln K},\Qp)$. Then $X_K$ has good reduction (i.e $\oln X$ is smooth), if and only if $V$ is a crystalline representation of $G_K:=\Gal(\oln K, K)$.
\end{theorem}

\begin{proof}
If $\oln X$ is smooth, that is if $X_K$ has good reduction, then this is the theorem of Faltings~\cite{MR1463696}. 

Now assume that $V$ is crystalline representation of $G_K$. Then $V$ is $B_{\cris}$-admissible, but the $B_{\cris}$-admissible representations are those semistable representations for which $N_{st}=0$ (remark~\eqref{rem:st-cris}).  So by theorem~\eqref{thm:MainTheorem} $\oln X$ is of type I, that is, $\oln X$ is smooth and so $X_K$ has good reduction.
\end{proof}

\bibliographystyle{alpha}
\renewcommand{\bibname}{References} 
 \bibliography{paper4arvix}

\begin{thebibliography}{{Wah}76}

\bibitem[AI12]{Andreatta:2012aa}
{Andreatta, Fabrizio} and {Iovita, Adrian}.
\newblock Semistable sheaves and comparison isomorphisms in the semistable
  case.
\newblock {\em Rend. Semin. Mat. Univ. Padova}, 128:131--285 (2013), 2012.

\bibitem[AIKct]{Andreatta:2015aa}
{Andreatta, Fabrizio}, {Iovita, Adrian}, and {Kim, Minhyong}.
\newblock A $p$-adic nonabelian criterion for good reduction of curves.
\newblock {\em Duke Mathematical Journal}, 164(13):2597--2642, 2015oct.

\bibitem[{Art}69]{Artin:1969wg}
{Artin, M}.
\newblock {Algebraic approximation of structures over complete local rings}.
\newblock {\em Inst. Hautes \'Etudes Sci. Publ. Math.}, 36(1):23--58, 1969.

\bibitem[{Ber}04]{MR2023292}
{Berger, Laurent}.
\newblock An introduction to the theory of {$p$}-adic representations.
\newblock In {\em Geometric aspects of {D}work theory. {V}ol. {I}, {II}}, pages
  255--292. Walter de Gruyter, Berlin, 2004.

\bibitem[{Bre}00]{MR1804530}
{Breuil, Christophe}.
\newblock Groupes {$p$}-divisibles, groupes finis et modules filtr\'es.
\newblock {\em Ann. of Math. (2)}, 152(2):489--549, 2000.

\bibitem[{Bre}eb]{Breuil:1997aa}
{Breuil, Christophe}.
\newblock Representations $p$ -adiques semi-stables et transversalite de
  griffiths.
\newblock {\em Mathematische Annalen}, 307(2):191--224, 1997Feb.

\bibitem[CF00]{Colmez:2000vq}
{Colmez, Pierre} and {Fontaine, Jean-Marc}.
\newblock Construction des repr\'esentations {$p$}-adiques semi-stables.
\newblock {\em Invent. Math.}, 140(1):1--43, 2000.

\bibitem[CIar]{Coleman:1999aa}
{Coleman, Robert} and {Iovita, Adrian}.
\newblock The frobenius and monodromy operators for curves and abelian
  varieties.
\newblock {\em Duke Mathematical Journal}, 97(1):171--215, 1999mar.

\bibitem[CL16]{Chiarellotto:2016ca}
{Chiarellotto, Bruno} and {Lazda, Christopher}.
\newblock Combinatorial degenerations of surfaces and {C}alabi-{Y}au
  threefolds.
\newblock {\em Algebra Number Theory}, 10(10):2235--2266, 2016.

\bibitem[{Cor}95]{MR1397057}
{Corti, Alessio}.
\newblock Recent results in higher-dimensional birational geometry.
\newblock In {\em Current topics in complex algebraic geometry ({B}erkeley,
  {CA}, 1992/93)}, volume~28 of {\em Math. Sci. Res. Inst. Publ.}, pages
  35--56. Cambridge Univ. Press, Cambridge, 1995.

\bibitem[{Del}70]{deligne14004equations}
{Deligne, Pierre}.
\newblock {\em \'equations diff\'erentielles \`a points singuliers
  r\'eguliers}.
\newblock Lecture Notes in Mathematics, Vol. 163. Springer-Verlag, Berlin-New
  York, 1970.

\bibitem[{Del}81]{Deligne:1981aa}
{Deligne, P.}
\newblock Rel\`evement des surfaces {$K3$}\ en caract\'eristique nulle.
\newblock In {\em Algebraic surfaces ({O}rsay, 1976--78)}, volume 868 of {\em
  Lecture Notes in Math.}, pages 58--79. Springer, Berlin-New York, 1981.
\newblock Prepared for publication by Luc Illusie.

\bibitem[{Fal}89]{MR1463696}
{Faltings, Gerd}.
\newblock Crystalline cohomology and {$p$}-adic {G}alois-representations.
\newblock In {\em Algebraic analysis, geometry, and number theory ({B}altimore,
  {MD}, 1988)}, pages 25--80. Johns Hopkins Univ. Press, Baltimore, MD, 1989.

\bibitem[{Fon}94]{10017177619}
{Fontaine, Jean-Marc}.
\newblock Le corps des p\'eriodes {$p$}-adiques.
\newblock {\em Ast\'erisque}, 223:59--111, 1994.
\newblock With an appendix by Pierre Colmez, P\'eriodes $p$-adiques
  (Bures-sur-Yvette, 1988).

\bibitem[{Fri}83]{Friedman:1983er}
{Friedman, Robert}.
\newblock {Global smoothings of varieties with normal crossings}.
\newblock {\em The Annals of Mathematics}, 118(1):75--114, 1983.

\bibitem[FS86]{Friedman:1986ve}
{Friedman, Robert} and {Scattone, F.}
\newblock {Type III degenerations of K3 surfaces}.
\newblock {\em Inventiones Mathematicae}, 83(1):1--39, 1986.

\bibitem[{Kat}89]{Kato:1989vk}
{Kato, Kazuya}.
\newblock Logarithmic structures of {F}ontaine-{I}llusie.
\newblock In {\em Algebraic analysis, geometry, and number theory ({B}altimore,
  {MD}, 1988)}, pages 191--224. Johns Hopkins Univ. Press, Baltimore, MD, 1989.

\bibitem[{Kat}94]{Kato:1994we}
{Kato, Kazuya}.
\newblock {Semi-stable reduction and $p$-adic \'etale cohomology}.
\newblock {\em Ast\'erisque}, 223:269--293, 1994.

\bibitem[{Kat}96]{Kato:1996wo}
{Kato, Fumiharu}.
\newblock {Log smooth deformation theory}.
\newblock {\em Tohoku Mathematical Journal}, 48(3):317--354, 1996.

\bibitem[{Kaw}94]{MR1269717}
{Kawamata, Yujiro}.
\newblock Semistable minimal models of threefolds in positive or mixed
  characteristic.
\newblock {\em J. Algebraic Geom.}, 3(3):463--491, 1994.

\bibitem[{Kaw}99]{MR1689354}
{Kawamata, Yujiro}.
\newblock Index 1 covers of log terminal surface singularities.
\newblock {\em J. Algebraic Geom.}, 8(3):519--527, 1999.

\bibitem[{Kul}77]{Kulikov:1977aa}
{Kulikov, Vik.~S.}
\newblock Degenerations of {$K3$} surfaces and {E}nriques surfaces.
\newblock {\em Izv. Akad. Nauk SSSR Ser. Mat.}, 41(5):1008--1042, 1199, 1977.

\bibitem[{Lan}03]{Landman:2010wp}
{Landman, Alan}.
\newblock {On the Picard-Lefschetz transformation for algebraic manifolds
  acquiring general singularities}.
\newblock {\em Transactions of the American Mathematical Society}, 181:89--126,
  2010-03.

\bibitem[LM18]{MR3699071}
{Liedtke, Christian} and {Matsumoto, Yuya}.
\newblock Good reduction of {K}3 surfaces.
\newblock {\em Compos. Math.}, 154(1):1--35, 2018.

\bibitem[LP80]{Looijenga:1980vy}
{Looijenga, Eduard} and {Peters, Chris}.
\newblock {Torelli theorems for K\"ahler $K3$ surfaces}.
\newblock {\em Compositio Mathematica}, 42(2):145--186, 1980.

\bibitem[{Mat}15]{MR3299851}
{Matsumoto, Yuya}.
\newblock Good reduction criterion for {K}3 surfaces.
\newblock {\em Math. Z.}, 279(1-2):241--266, 2015.

\bibitem[{Mau}ct]{Maulik:2014aa}
{Maulik, Davesh}.
\newblock Supersingular k3 surfaces for large primes.
\newblock {\em Duke Mathematical Journal}, 163(13):2357--2425, 2014oct.

\bibitem[{Mor}84]{Morrison:1984uh}
{Morrison, David~R}.
\newblock {The Clemens-Schmid exact sequence and applications}.
\newblock In {\em Topics in transcendental algebraic geometry (princeton, n.j.,
  1981/1982)}, pages 101--119. Princeton Univ. Press, Princeton, NJ, 1984.

\bibitem[{Nak}00]{Nakkajima:2000tl}
{Nakkajima, Yukiyoshi}.
\newblock {Liftings Of Simple Normal Crossing Log K3 And Log Enriques Surfaces
  In Mixed Characteritics}.
\newblock {\em J Algebraic Geometry}, 9:355 393, 2000.

\bibitem[{Nis}83]{Nishiguchi:1983wa}
{Nishiguchi, Kenji}.
\newblock {Degeneration of surfaces with trivial canonical bundles}.
\newblock {\em Japan Academy. Proceedings. Series A. Mathematical Sciences},
  59(7):304--307, 1983.

\bibitem[{Niz}08]{MR2372150}
{Nizio\l , Wies\l ~awa}.
\newblock Semistable conjecture via {$K$}-theory.
\newblock {\em Duke Math. J.}, 141(1):151--178, 2008.

\bibitem[{Oda}95]{MR1348768}
{Oda, Takayuki}.
\newblock A note on ramification of the {G}alois representation on the
  fundamental group of an algebraic curve. {II}.
\newblock {\em J. Number Theory}, 53(2):342--355, 1995.

\bibitem[{Ogu}18]{Ogus:2006tu}
{Ogus, Arthur}.
\newblock {\em {Lectures on Logarithmic Algebraic Geometry}}.
\newblock Cambridge University Press, 2018.

\bibitem[{Ols}03]{Olsson:2003uq}
{Olsson, M~C}.
\newblock {Universal log structures on semi-stable varieties}.
\newblock {\em Tohoku Mathematical Journal}, 2003.

\bibitem[{Ols}04]{Olsson:2004tk}
{Olsson, M~C}.
\newblock {Semistable degenerations and period spaces for polarized K3
  surfaces}.
\newblock {\em Duke Mathematical Journal}, 125(1):121--203, 2004.

\bibitem[{Par}94]{periodespadiques}
{Paris : Soc. Math.~de France, 1994.}, editor.
\newblock {\em P\'{e}riodes p-adiques: s\'{e}minaire de bures, 1988}.
\newblock Number 223 in Ast\'erisque. SMM, 1994.

\bibitem[{Per}77]{MR0466149}
{Persson, Ulf}.
\newblock On degenerations of algebraic surfaces.
\newblock {\em Mem. Amer. Math. Soc.}, 11(189):xv+144, 1977.

\bibitem[PP81]{Persson:1981wp}
{Persson, U.} and {Pinkham, H.}
\newblock {Degeneration of surfaces with trivial canonical bundle}.
\newblock {\em The Annals of Mathematics}, 113(1):45--66, 1981.

\bibitem[{Roh}09]{Rohde2009}
{Rohde, Jan~Christian}.
\newblock {\em An introduction to hodge structures and shimura varieties}.
\newblock Springer Berlin Heidelberg, Berlin, Heidelberg, 2009.

\bibitem[{Ser}79]{SerreLF:1979}
{Serre, Jean-Pierre}.
\newblock {\em Local fields}, volume~67 of {\em Graduate Texts in Mathematics}.
\newblock Springer-Verlag, New York, 1979.
\newblock Translated from the French by Marvin Jay Greenberg.

\bibitem[{Wah}76]{MR0422270}
{Wahl, Jonathan~M.}
\newblock Equisingular deformations of normal surface singularities. {I}.
\newblock {\em Ann. of Math. (2)}, 104(2):325--356, 1976.

\end{thebibliography}
\end{document}